\documentclass[10pt,a4paper]{article} %modify affilation?!!!! Federal State Budgetary Educational Institution of Higher Education "Saint-Petersburg State University?!!
\usepackage{amsmath,amssymb,amsthm,amsfonts,amscd,euscript,verbatim, t1enc, newlfont, graphicx, cancel, mathrsfs}
\usepackage{hyperref}
\usepackage[all]{xy}
\providecommand{\keywords}[1]{\textbf{\textit{Key words and phrases }} #1}

\theoremstyle{definition}

\theoremstyle{remark}

\theoremstyle{definition}

\numberwithin{equation}{subsection}

\newcommand\du{\underline{D}}
\newcommand\eu{\underline{E}}

\newcommand\bu{\underline{B}}

\newcommand\id{\operatorname{id}}

\DeclareMathOperator\co{\operatorname{Cone}}

\DeclareMathOperator\ndo{\operatorname{End}}

\DeclareMathOperator\inli{\varinjlim}

\newcommand\modd{\operatorname{Mod}}

\newcommand\ns{\{0\}}

\newcommand\perpp{{}^{\perp}}
\newcommand\opp{^{op}}

\usepackage{amsmath,amssymb,amsthm,graphicx}
\usepackage[T2A]{fontenc}
\usepackage[utf8]{inputenc}

\providecommand{\keywords}[1]{\textbf{\textit{Key words and phrases }} #1}

\usepackage{tikz} \usetikzlibrary{cd}

\newcommand{\xrar}{\xrightarrow}

\newcommand{\con}{\text{\,--\,}\mathrm{contra}}
\newcommand{\proj}{\mathrm{Proj}}
\newcommand{\projr}{{\mathrm{Proj}R\text{--}\operatorname{Mod}}}
\newcommand{\fprojr}{{\mathrm{FinProj}R\text{--}\operatorname{Mod}}}
\newcommand{\fproju}{{\mathrm{FinProj}U\text{--}\operatorname{mod}}}
\newcommand{\proju}{{\mathrm{Proj}U\text{--}\operatorname{mod}}}
\newcommand{\pr}{\mathrm{Pr}}

\newcommand{\Zbb}{\mathbb{Z}}
\newcommand{\tr}{\underline}
\newcommand{\K}{\mathscr{K}}

\newcommand{\PP}{\mathcal{P}}
\newcommand{\A}{\mathcal{A}}

\newcommand{\DE}{\mathscr{D}_E}
\newcommand{\ab}{\mathcal{A}b}
\newcommand{\SR}{S^{-1}R}

\newcommand{\End}{\mathrm{End}}

\newcommand{\obj}{\mathrm{Obj}\,}
\newcommand{\hw}{\tr{Hw}}
\newcommand{\Ht}{\tr{Ht}}

\newcommand{\lPr}{\langle\PP\rangle^{cl}}

\newcommand{\rmod}{R\,\text{--}\operatorname{Mod}}
\newcommand{\tif}{\tilde{F}}
\newcommand{\cu}{\tr{C}}
\newcommand\oo{{\pmb{1}}}
\newcommand{\z}{{\mathbb{Z}}}%{\mathBB{Z}}
\newcommand{\uu}{\mathcal{U}}
\DeclareMathOperator\adfu{\operatorname{AddFun}}
\DeclareMathOperator\adfupr{\operatorname{AddFun}_{\prod}}

\newtheorem{theorem}{Theorem}[subsection]
\newtheorem{proposition}[theorem]{Proposition}
\newtheorem{corollary}[theorem]{Corollary}
\newtheorem{lemma}[theorem]{Lemma}

\theoremstyle{definition}
\newtheorem{definition}[theorem]{Definition}
\newtheorem{note}[theorem]{Remark}
\numberwithin{equation}{subsection}

\begin{document}

%\subjclass{Primary 18E30  18E40 18G25; Secondary 14F05 55P42.} 
\title{Completing hearts of triangulated categories via weight-exact localizations}%On contramodules in weight-exact localizations} %Контрамодули и $t$-структуры в локализациях весовых категорий}
\author{Mikhail V.\,Bondarko, Stepan V.\,Shamov
%\address{St.Petersburg State University, 7/9 Universitetskaya nab., St. Petersburg 199034, Russia} \email{m.bondarko@spbu.ru}
\thanks{The work on the main results was supported by the Leader (Leading scientist Math) grant no. 22-7-1-13-1 of  the Theoretical Physics and Mathematics Advancement Foundation «BASIS».   %{№22-7-1-13-1}
 %} thanks{
 %The research of Stepan V.\,Shamov was carried out at the Leonhard Euler International Mathematical Institute in Saint Petersburg and was supported by the  Ministry of science and higher education of the Russian Federation  (agreement № 075–15–2022–287, signed on 06.04.2022)
 The work of S. Shamov was performed at the Saint Petersburg Leonhard Euler International Mathematical Institute and supported by the Ministry of Science and Higher Education of the Russian Federation (agreement no. 075–15–2022–287).}}\maketitle
%\date{\today}

%\usepackage{hyperref,verbatim} %\usepackage{enumitem}
%\setlist{noitemsep}
%УДК: 512.58+512.582.2+512.585
%\medskip
%\section*{Abstract}
\begin{abstract} 
We study a weight-exact localization $\pi$ of a well generated (in particular, of a compactly generated) triangulated category $\cu$  along with the embedding of the hearts of adjacent $t$-structures coming from the functor right adjoint to %these localizations
 $\pi$.  For a wide class of weight-exact localizations we prove that the functors relating the corresponding four hearts are completely determined by the heart $\hw$ of the %initial 
  weight structure on $\cu$ along with the set of %morphisms in it 
 $\hw$-morphisms that we invert via $\pi$; it also suffices to know the %(fully faithful exact) functor between
   corresponding embedding of the hearts of $t$-structures.
 Our main results vastly generalize %the % relations between weight-exact localizations and non-commutative localizations of rings that were described in an 
  the description of non-commutative localizations of rings in terms of weight-exact localizations given in an earlier paper of the first author. That paper was essentially devoted to weight-exact localizations by compactly generated subcategories, whereas in the current text we %concentrate 
   focus on "more complicated" localizations.

 We recall that two types of %examples with %the initial category being
 localizations of the sort we are interested in   %(the derived category of the ring $R$)
   were studied by L. Positselsky and others. They took  $\cu=D(\rmod)$;   the heart of the first $t$-structure (the canonical $t$-structure on $\cu$) was equivalent to $\rmod$, and the second heart was equivalent to the exact abelian category $\mathcal{U}\con\subset \rmod$ of $\uu$-{\it contramodules} (corresponding to a set of $\projr$-morphisms  $\mathcal{U}$ %coming either from
 related either to a homological ring epimorphism $u:R\to U$ or to an ideal of $I$ of $R$). The functor $\rmod\to \mathcal{U}\con$ induced by %the localization
 $\pi$  is a certain completion one (and in the case where $I$ is a finitely generated ideal of a commutative ring $R$ the corresponding derived completions were %mentioned
   studied  in an important paper of Dwyer and Greenlees).
  Consequently, the hearts of the corresponding  weight structures are equivalent to $\projr$ and to  %corresponding category of projective  
 the subcategory of projective objects of $\mathcal{U}\con$, respectively. %Our results imply that 
 Moreover, the connecting functors between these categories are isomorphic to ones coming from any weight structure {\it class-generated} by a single compact object whose endomorphism ring is $R^{op}$; in particular, one can take  $R=\z$ and $\cu=SH$ and re-prove some  %well-known 
 important statements due to Bousfield.  
\end{abstract}
%???????!!!!!!!\subjclass{ 18G80, 14F08 (Primary)  18F20, 18G05, 18E10, 14A15, 14G40 (Secondary)} %, 16E65


\tableofcontents

\section{Introduction}

\subsection{History and motivation}\label{shist}

In an %a %previous  preceding
 earlier  paper \cite{BoS18} it was demonstrated that the so-called {\it non-commutative ring localizations} of a ring $R$ %(that corresponds to  inverting 
 (that is,  ring homomorphisms that universally invert a set $S_0$ of morphisms between finitely generated projective $R$-modules; they are called {\it universal localizations} in  \S6 of \cite{KrS10}) are closely related to weight-exact localizations.

Let us explain this in some detail; the reader may skip this and proceed to \S\ref{sform} below. %motivational part.  %In this case one can take 
 Take the triangulated category $\cu^0=%D^{perf}(\rmod)=
 K^b(\fprojr)$ (the homotopy category of  bounded complexes of finitely generated projective $R$-modules),   localize $\cu^0$ by its triangulated subcategory $\du^0$ generated by cones of the elements of $S_0$, and %consider the endomorphism ring
 then the non-commutative localization in question is given by  $U=(\End^{\tr{E}^0}(\pi^0(R)))^{op}$; %of %the image %of 
%$\pi(R)$ %in %the localized category  %for this localization 
(here $\pi^0:\cu^0\to \eu^0=\cu^0/\du^0$ is our localization functor). A simple (and not really interesting) case here is $S_0=\{R\stackrel{s\id_R}%{\to}
{\longrightarrow}R,\ s\in S\}$, where $S$ %runs through 
 is a (say) multiplicative subset of a commutative $R$; then $U=\SR$ and $\eu^0$ can be embedded into $K^b(\fproju)$.  %$ D^{perf}(U\text{-}\mathrm{mod})$. 
%cite whom??? Dwyer?! A. Neeman and A. Ranicki, Noncommutative localisation in algebraic K-theory I?!

Next, recall that weight structures (see Definition \ref{dwso} below)  were independently introduced in \cite{Bon10} and \cite{konk}; they are certain counterparts of $t$-structures  that are defined somewhat similarly and are also closely related to them. %???? (and we will soon discuss this relation).
  The category $\cu^0=K^b(\fprojr)$ is endowed with a certain weight structure $w^0=(\cu_{w^0\le 0},\cu_{w^0\ge 0})$ whose heart $\hw^0=\cu_{w^0\le 0}\cap \cu_{w^0\ge 0}$ is equivalent to $\fprojr$; though $\eu^0$ does not have to embed into %$ D^{perf}(U\text{-}\mathrm{mod}$
 $K^b(\fproju)$, it is  endowed with a weight structure $w^{\eu^0}$ %such that 
whose heart $\hw^{\eu^0}$ embeds into %projective $U$-modules. 
 $\fproju$. More generally, one can localize {\it any} weighted triangulated category $\cu^0$ (that is, $\cu^0$ is endowed with a weight structure $w^0$) such that $\hw^0%\cong
  = \fprojr$  by the triangulated subcategory $\du^0$ generated by cones of $S_0$ (in $\cu^0$) and obtain  that $U=(\End^{\tr{E}}(\pi^0(R)))^{op}$ and $\eu^0$ is endowed with a with a weight structure  "compatible with $w$" such that $\hw^{\eu^0}$ embeds into %projective $U$-modules. 
 $\fproju$; cf. Theorem 4.2.3 of \cite{BoS18}. %  ibid.
Moreover, %the theory of 
 certain "triangulated" arguments inspired by the theory of weight structures  naturally yield a method of computing this $U$; see  \S3 of ibid. 

Consequently, one may use the category  $K^b(\fprojr)$ and its localizations for calculating noncommutative localizations of $R$ and also for computing certain full subcategories of various localized categories; note here that there exist weight structures with "rather simple" hearts on various stable homotopy categories (see Theorem 4.1.1 of \cite{Bon21} and \S\ref{Stable} below). We would like to note that computing these hearts using just the definition of Verdier localizations appears to be a much more difficult task.

Now let us add adjoint functors and $t$-structures to this picture. Our %"ordinary localization" 
 $S$-localization example above demonstrates that one cannot hope for a right adjoint to $\pi^0$ in the case  $\cu^0= K^b(\fprojr)$; it also appears that one cannot find $t$-structures related to our computations (in any reasonable way). Yet one can overcome these difficulties by means of passing to large(r) categories. The (rather easy) Proposition \ref{w^PP} below demonstrates that the category $\cu=D(\rmod)$ (the unbounded derived category of the category of all $R$-modules) is naturally endowed with a weight structure $w^{\pr}$ such that $\hw^{\pr}$ is equivalent to $\projr$ (and consists of complexes isomorphic to projective modules put in degree $0$).  Next one takes $\du$ to be the {\it localizing} subcategory of $\cu$ generated by cones of $S_0$, that is, $\du$ is the smallest strictly full triangulated subcategory of $\cu$ that is closed with respect to $\cu$-coproducts. Then %(this "version" of)  
 $\eu=\cu/\du$ is endowed with a weight structure compatible with $w^{\pr}$ such that $\hw^{\eu}\cong \proju$; this statement is easily seen to be a particular case of Theorem 4.3.1.4(2) of \cite{BoS18}. Moreover,  for the localization functor $\pi:\cu\to \eu$ there exists a (fully faithful exact) right adjoint functor $\pi_*$ %:\eu\to \cu$ 
 that is $t$-exact with respect to {\it adjacent} $t$-structures $t^{\eu}$ and $t^{can}$ (see  Theorem 4.3.1.4(1,4) of  loc. cit. and the easy Proposition \ref{w^PP}(\ref{class-wPP}) below). %ibid.
The latter statement (in contrast to the more general formulation in loc. cit.)
 is easily seen to %be equivalent to Theorem 
 follow from Proposition 3.4 of  \cite{Dwy06} (cf. Remark \ref{rdwy} below); note that the latter paper essentially relates non-commutative localizations of $R$ to localizations of the sort that we have just described. %cf. Remark \ref{rpos}?? below????!!!
Moreover, Theorem 4.3.1.4(1,4,5) of \cite{BoS18} immediately implies that for any weighted category $\cu$ endowed with a weight structure {\it class-generated} by a compact object $P$ with $R=(\End^{\tr{C}}(P))^{op}$ one can also construct a similar localization functor, its right adjoint, and the corresponding weight and $t$-structures, and the heart $\Ht^{\eu}$ is equivalent to the category of $U$-modules (cf. Theorem \ref{thea} below). Once again, we obtain a statement that describes a full subcategory of a localization that may be "quite complicated".

The main goal of the present paper is to generalize and extend these results; in particular, we wanted to allow $S_0$ to consists of morphisms between arbitrary projective $R$-modules. We also mention  soon (in Remark \ref{rrpos} below)
 certain examples related to this setting; their existence in the literature appears to increase the value of this paper significantly (and was quite surprising to the authors). 

 Note that (in contrast to the setting of \cite[Theorem 4.3.1.4]{BoS18}) it does not make much sense to assume that $\cu$ is {\it compactly generated} since %in this case 
 the localized category $\eu$ does not have to possess this property anyway. For this reason (and for the sake of generality) we only demand $\cu$ to be {\it well generated} (and just assume $w$ to be {\it smashing}); hence our results are more general than loc. cit. even in the case where $S_0$ consists of  $\fprojr$-morphisms. 

\subsection{Main general results of the paper and a discussion of examples}\label{sform}

Unfortunately, the reader would have to find all the definitions needed to understand the following  statements %theorem  %(most 
(all of which except %the last statement in its part \ref{isqa}
 part \ref{idet} will be proved in \S\ref{sweadjt}) in \S\ref{Prelim} below. %the main body of the paper.
%We also omit some of the detail in its formulation. %See??????!!!!
 This theorem combines %almost all of 
 our main general results.

\begin{theorem}\label{thea}
Assume that  $\tr{C}$ is a well generated triangulated category endowed with a smashing weight structure $w$
%(see Definition \ref{generated}(\ref{iwell})) 
 and $\tr{D}$ is its localizing subcategory  generated by a set  $\mathcal{U}_0\cup\mathcal{U}_1\cup\mathcal{U}_2$, where $\mathcal{U}_0$ consists of cones of a set $S_0$ of $\hw$-morphisms, the elements of $\mathcal{U}_1$ are left $w$-degenerate, and the  elements of   $\mathcal{U}_2$ are right $w$-degenerate.
\begin{enumerate}
\item \label{iexi}
Then  the Verdier localization  $\tr{E}=\tr{C}/\tr{D}$ is   locally small %discuss in ls remark??!
 and the localization functor  $\pi: \tr{C}\to\tr{E}$ %respects coproducts and 
 possesses a fully faithful exact right adjoint functor  $\pi_*$.

\item\label{wwwa}  %Then 
There exists a  weight structure ${w^{\tr{E}}}$ on  $\tr{E}$ such that $\pi$  is weight-exact. 

\item %\label{t,t_E} 
 There exist   $t$-structures $t$ and $t^{\tr{E}}$ on $\cu$ and  $\tr{E}$, respectively, that %is 
 are adjacent to  $w$ and $w^{\tr{E}}$, and $\pi_*$ is $t$-exact with respect to them.

Moreover, %certain
 the  Yoneda-type functor $Y: \Ht\to\adfu(\hw^{op},\ab)$,
 $M\mapsto \tr{C}(-,M)$, (resp. its $\Ht^{\tr{E}}$-analogue) gives an equivalence of  $\Ht$ (resp. of $\Ht^{\tr{E}}$) with the full exact abelian subcategory $\adfupr(\hw^{op},\ab)$ of $\adfu(\hw^{op},\ab)$ (resp.  $\adfupr(\hw^{\tr{E},op},\ab)$ of $\adfu(\hw^{\tr{E},op},\ab)$) that consists of those functors that respect small products (cf. Remark  \ref{rsmash}(2) below).

\item \label{ichar} 
 $\pi_*$ restricts to an equivalence of $\Ht^{\tr{E}}$ with the full exact subcategory of   $\Ht$ that is characterized by any of the following conditions.
\begin{enumerate}
\item $N\in\obj\Ht$ and  there are no non-zero morphisms from $ (\mathcal{U}_0\cup\mathcal{U}_0[-1])$ into $N$;
\item $N\in\obj\Ht$ and $ H_N(s)=\tr{C}(s,N)$ is  bijective for any $s\in S_0$.
\end{enumerate}
\item\label{isqa} 
The diagram 
	\begin{equation}\label{ecomma}
	\begin{CD}
 \hw @>{H_t^{\hw}}>> \Ht\\
@VV{\pi^{\hw}}V@VV{\tr{H\pi}_*^*}V \\
\hw^{\tr{E}} @>{H_{t^{\tr{E}}}^{\hw^{\tr{E}}}}>>\Ht^{\tr{E}}
\end{CD}
\end{equation}
	is commutative up to an isomorphism; here $H_t^{\hw}$,  $H_{t^{\tr{E}}}^{\hw^{\tr{E}}}$, and ${\pi^{\hw}}$ are the corresponding restrictions of the (two 
	  $t$-structure cohomology) functors  $H_t$,  $H_{t^{\tr{E}}}$, and  $\pi$, respectively,  and $\tr{H\pi}_*^*$ is the left adjoint to the restriction of $\pi_*$ to $\Ht$; cf. Theorem \ref{pi_&_G}(\ref{icomm}) below for some additional detail on these functors.
	
Moreover, 	$H_t^{\hw}$ (resp. $H_{t^{\tr{E}}}^{\hw^{\tr{E}}}$) gives an equivalence of $\hw$ (resp. $\hw^{\tr{E}}$) with the full subcategory of projective objects of the abelian category  $\Ht$ (resp. $\Ht^{\tr{E}}$). 

\item\label{idet} The diagram (\ref{ecomma}) is essentially canonically determined  by the couple $(\hw,S_0)$  as well as by $(\Ht,\pi_*(\Ht^{\eu}))$; cf. Theorem \ref{tp01ext}(2), its proof, and Remark \ref{rdetht} for more detail on this statement.

\item\label{iessi}
 %Moreover, if  
 If $\mathcal{U}_1=\mathcal{U}_2=\varnothing$ then
 $M\in\obj\tr{C}$ belongs to the essential image class $\pi_*(\tr{E})$ if and only if $H_t(M[i])\in \pi_*(\Ht^{\tr{E}})$ %(see Proposition~\ref{sic!t}(\ref{H^t}))
 for all $i\in\Zbb$. % (cf. the next assertion for the description of the latter essential image).

%Here $H_t:\cu\to \Ht$ is the homological functor (of zeroth $t$-homology) corresponding to $t$.
\end{enumerate}\end{theorem}

\begin{note}\label{rrpos}
\begin{enumerate} 
\item %It appears that in
Exact abelian subcategories %of $\rmod$ characterized by the conditions described in
of abelian categories that can be characterized by conditions similar to that in Theorem \ref{thea}(\ref{ichar})  %(for some set $S_0$ of $\rproj$-morphisms) 
 were studied in detail in the important paper \cite{GeL91}.  %pd\le 1 or not??????!!!
Note also that most of that paper is devoted to abelian subcategories of $\rmod$ that can be characterized this way; thus they can be studied by means of applying Theorem \ref{thea} to $\cu=D(\rmod)$ (see %either \S\ref{shist} or 
Proposition  1.1 of ibid. and Proposition  \ref{w^PP}(\ref{class-wPP}) below).
Moreover, Proposition  3.8 of ibid. is closely related to  non-commutative localizations and the aforementioned isomorphism $\Ht^{\eu}\cong U-\modd$ (in the case $\cu=D(R-\modd$).

\item %In most of the concrete examples of this paper we will consider the case
 In the more concrete statements of this paper we will usually assume $\hw\cong \projr$ and $\Ht\cong\rmod$ (for some ring $R$). 

%\item 
The reason for this is that in most of %"actual" 
%examples of localizations of the sort described in
 earlier papers related to  Theorem \ref{thea} %that were considered by other authors 
 %in existing literature
 only localizations of the category $D(\rmod)$ were considered (yet cf. \S\ref{Stable} below). Note however that (by the virtue of  Theorem \ref{thea}) %(\ref{isqa})) 
 the computations in these papers also give much information on localizations of general well generated weighted triangulated categories with $\hw\cong \projr$.

\item We will concentrate on two (more or less) concrete sorts of morphism sets $S_0$ that yield examples that appear to be both (the most) "popular" and important. Both of them were called certain categories of contramodules in \cite{BaP20}, \cite{Pos16}, \cite{Pos17}, and \cite{Pos18}; we will define them in  \S\ref{Bazz} below. 

Yet the term {\it $I$-complete modules} (for a finitely generated ideal $I$ of a commutative ring $R$) that is a synonym for {\it $I$-contramodules} was used in \cite{DwG02} (cf. Remark \ref{rpos}(\ref{idwg} below), and the left adjoint to the corresponding embeddings of hearts (see  Theorem \ref{thea}(\ref{isqa}))   is a certain $I$-completion functor. %Thus it makes sense to call the functor $\pi$ a
 This motivated the title of our paper.
%In the case where $S_0$ does not consist (only) of $\fproj$-morphism, concrete examples of "our" localizations of $D(\rmod)$  known to the author may be divided into two (rather rich and interesting) families. The first of them were called {\it $U$-contramodules} in \cite{BaPo20}
\end{enumerate}
\end{note}

%Let us  now describe the contents  of the paper. 
%Some more information of this sort may be found in the beginnings of sections. %??????

\subsection{On the contents of the paper}.

In~\S\ref{Prelim} we recall some basics on  triangulated categories and weight and $t$-structures on them; we also prove the somewhat new  Proposition \ref{p01}. % on isomorphism classes of objects of weight $[0,1]$ (in various weighted categories with equivalent hearts).

In~\S\ref{sgenwe} we prove (most of) our main general statements on weight-exact localizations; they yield Theorem \ref{thea}(\ref{iexi}--\ref{isqa},\ref{iessi}).%!!!! 
We also describe some criteria for the weight-exactness of the localization $\cu\to \cu/\du$ in the case where $\du$ is semi-simple; %they will be applied in 
 we will need them for constructing certain (non)-examples in \S\ref{scount} only.

In~\S\ref{Loc_D(R)} we mostly study weight structures generated by a single compact object.  In the case %where 
 $\cu=D(\rmod)$ and for  two concrete types of the sets $S_0$ we obtain localization functors studied in several papers of L. Positselsky and others; the corresponding categories $\Ht^{\eu}$ are called categories of (certain) contramodules. We also prove (in Theorem \ref{tp01ext}) that weight-exact localizations of well generated categories $\cu$ and $\cu'$ with equivalent hearts are "closely related".  
We apply this observation to well generated weighted categories with $\hw\cong \projr$.

In ~\S\ref{Examples} we discuss one of the sources of  ($u-$)contramodules, that is, Ore localizations of rings by countable multiplicative subsets (see Proposition \ref{pore} for  the detail). We prove that this countability assumption cannot be lifted. %We also
 Lastly, we discuss certain (monogenic) stable homotopy examples for our results.

Note also that some information of the contents of our sections can be found in their beginnings.

\section{Preliminaries}\label{Prelim}

The main goal of this section is to recall some definitions and fact related (mostly) to triangulated categories and weight and $t$-structures on them. Probably, the most original statement of this section is Proposition \ref{p01} on isomorphism classes of objects of weight $[0,1]$ (in various weighted categories with equivalent hearts).

\subsection{Some notation and definitions}\label{general}
%??Введём необходимые определения и обозначения, следуя \cite{Bon21}, \cite{BoS19} и~\cite{Kra10}.
Let us intoduce some definition, conventions, and notation.

\begin{definition}\label{notation}\hspace{1cm}
\begin{itemize}
\item In this paper we will mostly consider locally small categories. The only exceptions that we will need are some categories of functors, and we do not need much of them. 

%subcategories vs. classes of objects?!

\item %Для (локально малой)???? категории 
 For a (locally small) category $C$ and ${X,Y\in \mathrm{Obj} C}$ we will write ~$C(X,Y)$ for the set of $C$-morphisms from ~$X$ into~$Y$. 

\item  $\tr{C}$ will always denote a triangulated category, 
 $\tr{D}$ is its triangulated subcategory, and  $\tr{E}$ will denote the Verdier localization~$\tr{C}/\tr{D}$.

\item For $M,N\in\mathrm{Obj}\tr{C}$ we will write  $M\perp N$ if ${\tr{C}(M,N)=\{0\}.}$
 For $X,Y\subset\mathrm{Obj}\tr{C}$ we write ${X\perp Y,}$ whenever ${M\perp N}$ for all 
 ${M\in X, N\in Y.}$  For $\PP\subset \mathrm{Obj}\tr{C}$ we will use the notation $\PP^\perp$ for the class  ${\{N\in\mathrm{Obj}\tr{C}:\forall M\in\PP\enskip M\perp N\}.}$

\item For any ${A,B,C\in\mathrm{Obj}\tr{C}}$ the object $C$ is said to be an  \textit{extension of
 $B$ by means of $A$} if there exists a distinguished triangle  ${A\to C\to B\to A[1]}$.

\item  $\tr{C}$ is said to be \textit{idempotent complete} if any idempotent endomorphism in it yields a direct sum decomposition.
\end{itemize}
\end{definition}

\begin{lemma}[{\cite[Proposition 1.6.8]{Nee01}}] \label{l168}
If all countable coproducts exist in $\cu$ %is stable with respect to countable coproducts?????!
then $\cu$ is idempotent complete. \end{lemma}

We also need a few definitions related to coproducts and products.

All coproducts and products in this paper will be small.

\begin{definition}\label{Brown_Rep}
Let %$\cu$ be a triangulated category.
 $\tr{B}$ and $\tr{B}'$ be additive categories.
\begin{enumerate}
\item\label{ismash}  
%An additive category   
 $ \tr{B}$ is said to be {\it smashing} if it is closed with respect to (small) coproducts.

Moreover, if this is the case then we will say that a class $\PP\subset\tr{B}$ is {\it smashing} if it is closed with respect to (small) $\bu$-coproducts.

\item\label{iadfupr} %For additive categories $\tr{B}$ and $\tr{B}'$
 We will write 
 $\adfupr(\tr{B},\tr{B}')\subset\adfu(\tr{B},\tr{B}')$ for the subcategory of those additive functors that respect  %$\tr{B}$-
  products.
\end{enumerate}
\end{definition}

\begin{note}\label{rsmash}
1. Obviously, $\adfupr(\tr{B},\tr{B}')$ is an exact abelian subcategory of the category $\adfu(\tr{B},\tr{B}')$.

2. If $\tr{B}=C^{op}$ then a functor $\tr{B}\to \tr{B}'$ respects products if and only if the corresponding contravariant functor from $C$ into $\tr{B}'$ converts all (small) coproducts that exist in $C$ into $\tr{B}'$-products.
\end{note}

Now we relate these definitions to triangulated categories.

\begin{definition}\label{generated}
 Let $\cu$ be a triangulated category. 
\begin{enumerate}
\item\label{homolog} %Пусть $\tr{C}$\,---\,триангулированная категория,
 For an abelian category $\A$, we say that an additive functor  ${H:\tr{C}\to\A}$ is  \textit{homological}  if $H$ converts distinguished triangles into long exact sequences. 

If this is the case, then the corresponding contravariant functor $H'$ from $\tr{C}^{op}$ into $\A$
 is said to be \textit{cohomological}.

\item\label{B_R} $ \tr{C}$ is said to satisfy the {\it Brown representability property} if it is smashing and
 any cohomological functor %$\opp\to \ab$ 
%that respects products (that is, ) is representable.
in $\adfupr(\tr{C}^{op},\ab)$ (cf. Remark \ref{rsmash}(2)) is representable (by an object of $\cu$).

\item\label{localiz}  If  $\tr{C}$ is smashing and  ${\PP\subset\mathrm{Obj}(\tr{C})}$ then we will write  $[\PP]^{cl}$  for the smallest smashing class of objects of $\tr{C}$ that is closed with respect to extensions and contains ~$\PP.$

We will write ~${\lPr_{\tr{C}}}$ for the full subcategory of $\tr{C},$ whose object class equals  ${[\bigcup_{i\in\Zbb}\PP[i]]^{cl},}$
 and call it the \textit{localizing subcategory of $\tr{C}$ generated by~$\PP$.}

\item Assume that  an infinite  cardinal $\alpha$ either equals $\aleph_0$ or the successor cardinal $\alpha'+1$ for some cardinal $\alpha'$. Then we will say that a {\bf set}  $\PP\subset \obj \tr{C}$ is {\it $\alpha$-perfect} whenever $\tr{C}$ is smashing %, and for an infinite  cardinal $\alpha$ that either equals $\aleph_0$ or the successor cardinal $\alpha'+1$ for some cardinal $\alpha'$ 
 and the functor $H^{\PP}=\coprod_{P\in \PP}\tr{C}(P,-),$ ~$\tr{C}\to \ab$ satisfies the following conditions for any set $I$ and $N_i\in \obj \cu$: $$H(\coprod_{i\in I}N_i)\cong \inli_{J\subset I,\ \operatorname{card}J<\alpha}H(\coprod_{i\in J}N_i),$$ and  the homomorphism $H(\coprod_{i\in I} f_i)$ is surjective whenever all $H(f_i)$ are (where $f_i$ are $ \tr{C}$-morphisms). In the case $\alpha=\aleph_0$ the elements of   $\PP$ are said to be {\it compact} (cf. Remark \ref{rwell}(1) below).%\footnote{Obviously, an object is compact if and only if the functor ${H^M=\tr{C}(M,-):\tr{C}\to \ab}$ respects coproducts.}

 \item\label{iwell} We will say that $\tr{C}$ is {\it $\alpha$-well generated}  if there exists an $\alpha$-perfect set $\PP\subset \obj \tr{C}$ such that ${\tr{C}=\lPr_{\tr{C}}.}$ %In the case $\alpha=\aleph_0$ 

$\tr{C}$ is  said to be  {\it  well generated} (resp. {\it compactly generated}) if it is  $\alpha$-well generated  for some $\alpha$ as above (resp. for $\alpha=\aleph_0$).
\end{enumerate}
\end{definition}

\begin{note}\label{rwell}
1. Obviously, %an object 
$P\in \obj \cu$ is compact if and only if the functor ${H^M=\tr{C}(M,-):\tr{C}\to \ab}$ respects coproducts; this is also equivalent to the $\aleph_0$-perfectness of the set $\{P\}$.

2. Our definition of an $\alpha$-perfect set is easily seen to be equivalent to conditions (G2) and (G3) of \cite{krauwg}. Applying Proposition 8.4.1 of %\cite{neebook} 
\cite{Nee01} we obtain that our definition of well generation is equivalent to the definition given in %loc. cit.  
 \cite{krauwg}. By Theorem A of loc. cit. it is also equivalent to the (original) Definition 8.1.6 of  \cite{Nee01}. %{neebook}.
\end{note}

\subsection{Weight structures: basics}\label{w-str}

Recall that $\cu$ and $\cu'$ will always stand for some triangulated categories.

\begin{definition}\label{dws}%\hspace{1cm}
 
 Two subclasses ${\tr{C}_{w\leq0},\tr{C}_{w\geq0}\subset\mathrm{Obj}\tr{C}}$ will be said to define a {\it weight
structure} $w$ on 
  $\tr{C}$ (and $(\tr{C},w)$ is said to be a {\it weighted category}) whenever they satisfy the following assumptions: 
\begin{enumerate}
\item $\tr{C}_{w\leq0},\tr{C}_{w\geq0}$ are closed with respect to passing to direct summands;
\item ${\tr{C}_{w\leq0}\subset \tr{C}_{w\leq0}[1]}$ and ${\tr{C}_{w\geq0}[1]\subset\tr{C}_{w\geq0}}$;
\item\label{aort} ${\tr{C}_{w\leq0}\perp\tr{C}_{w\geq0}[1]};$
\item For any  ${M\in\mathrm{Obj}\tr{C}}$ there exists a \textit{weight decomposition} distinguished triangle ${LM \to M \to RM \to LM[1]},$ where ${LM\in\tr{C}_{w\leq0}}$ and ${RM\in\tr{C}_{w\geq0}[1]}$.
\end{enumerate}
\end{definition}

We will also need the following definitions related to weight structures.

\begin{definition}\label{dwso}%\hspace{1cm}

Assume that $w$ is a weight structure on $\cu$.

\begin{enumerate}

\item The full subcategory  $\tr{Hw}\subset \tr{C}$ whose object class equals 
 ${\tr{C}_{w=0}=\tr{C}_{w\geq0}\cap\tr{C}_{w\leq0},}$ is called the  \textit{heart} of $w$.

\item  $\tr{C}_{w\geq i}$ (resp. $\tr{C}_{w\leq i},$ $\tr{C}_{w=i}$) will denote the class
 $\tr{C}_{w\geq 0}[i]$ (resp. $\tr{C}_{w\leq 0}[i],$ $\tr{C}_{w=0}[i]$).

\item We say that $M\in\mathrm{Obj}\tr{C}$ is  \textit{left} (resp. \textit{ right) $w$-degenerate} %(or weight-degenerate if the choice of a weight structure is clear) 
if  ${M\in\bigcap_{i\in\Zbb}\tr{C}_{w\geq i}}$ (resp. ${M\in\bigcap_{i\in\Zbb}\tr{C}_{w\leq i}}$).
	
	$w$ is said to be {\it non-degenerate} if ${\bigcap_{i\in\Zbb}\tr{C}_{w\geq i}=\bigcap_{i\in\Zbb}\tr{C}_{w\leq i}=\{0\}}$.
	
	\item For $i\le j\in \z$ we will write $\cu_{[i,j]}$ for the class $\cu_{w\le j}\cap \cu_{w\ge i}$.

\item  If $(\tr{C}',w')$ is also a weighted category then an exact functor $F:\tr{C}\to \tr{C}'$ %between weighted categories $(\tr{C},w)$ and $(\tr{C}',w')$
  is said to be
 %$F$~
  \textit{weight-exact} if ${F(\tr{C}_{w\leq0})\subset \tr{C}'_{w'\leq0}}$
 and~${F(\tr{C}_{w\geq0})\subset \tr{C}'_{w'\geq0}}$.

\item\label{idescends} Let $\tr{E}$~be a Verdier localization of a weighted category $(\tr{C},w)$. % by a (triangulated) subcategory $\tr{D}$. 
 We say that $w$ \textit{descends} to $\tr{E}$ if there exists a  weight structure $w^{\tr{E}}$ of $\tr{E}$ such that the functor  $\pi:\tr{C}\to\tr{E}$ is weight-exact. 
 
 \item\label{iconn} A class $\PP\subset \obj\cu'$ is said to be {\it connective} (in $\cu'$) if $\PP\perp \cup_{i>0}\PP[i]$.
 %весо-точен.
\end{enumerate}
\end{definition}

\begin{note}\label{rdescends}
If $w$ descends to $\tr{E}$ for $\tr{C},w,\pi$, and $\tr{E}$ as in Definition \ref{dws}(\ref{idescends}) then 
 Proposition 3.1.1(1) of  \cite{BoS19} gives the following description of $w^{\tr{E}}$: the class $\eu_{w^{\tr{E}}{\le 0}}$ (resp. $\eu_{w^{\tr{E}}{\ge 0}}$) is the class of all direct summands of elements of $\pi(\cu_{w\le 0})$ (resp. $\pi(\cu_{w\ge 0})$) in $\tr{E}$. 
\end{note}

\begin{definition}\label{w-generated} %\hspace{1cm}
Let  $(\tr{C},w)$ be a weighted category, $\PP\subset\mathrm{Obj}\tr{C}$.
\begin{itemize} 
\item We will say that $w$ is \textit{smashing} if both %the category~
 $\tr{C}$  and the class~$\tr{C}_{w\geq0}$ are smashing.

\item We say that %the class $\PP\subset\mathrm{Obj}\tr{C}$
  $\PP$ \textit{class-generates}~$w$ if $\tr{C}$ is smashing, ${\tr{C}_{w\geq0}=[\bigcup_{i\geq0}\PP[i]]^{cl}}$, and~${\tr{C}_{w\leq0}=[\bigcup_{i\leq0}\PP[i]]^{cl}.}$\footnote{Then $\PP$ is easily seen to be connective; see Lemma \ref{ll168}(3) below.}

%\item We say that $\PP$ \textit{generates}~$w$ if  ${\tr{C}_{w\geq0}=(\bigcup_{i<0}\PP[i])^\perp.}$
\end{itemize}
\end{definition}

\begin{note}\label{sic!w}\hspace{1cm}
\begin{enumerate}
\item\label{nota} Weight structures were independently introduced in \cite{konk} (where they were called co-$t$-structures) and in~\cite{Bon10}. In loc. cit. %the following notation was used: 
 the classes  $\tr{C}_{w\geq0}$ and $\tr{C}_{w\leq0}$ were denoted by  $\tr{C}^{w\leq0}$ and $\tr{C}^{w\geq0}$, respectively.

\item For an additive category ~$\tr{B}$ one can (see  Remark 1.2.3(1)  of \cite{BoS18b}) define the so-called \textit{stupid} weight structure~${w^{st}}$  on the homotopy category~${\K^*(\tr{B})}$  of complexes (where $*$ stands either for  $b,+,-,$ or $\varnothing$)
 by setting ${\K^*(\tr{B})_{w^{st}\leq 0}}$
 (resp. ${\K^*(\tr{B})_{w^{st}\geq 0}}$) to be the class of complexes homotopy equivalent to those ones whose terms in negative (resp. positive) degrees are zero.  %delete????
% Весовое разложение комплекса $X^\bullet$ % можно построить, пользуясь его \textit{<<глупой>>} фильтрацией (см.\cite[III.7.5]{GeM03}).
\end{enumerate}
\end{note}

For some of our formulations, we will need the following relation of  "objects of weight $[0,1]$" to cones of $\hw$-morphisms.

\begin{proposition}\label{p01}
Assume that $(\tr{C},w)$ is a weighted category.

1. %Then the class $\tr{C}_{[0,1]}=\tr{C}_{w\geq0}\cap \tr{C}_{w\leq1}$ consists precisely of cones of morphisms between elements of  $\tr{C}_{w=0}$.
For $U\in \obj \tr{C}$ we have $U\in \tr{C}_{[0,1]}(=\tr{C}_{w\geq0}\cap \tr{C}_{w\leq1})$ if and only if $U$ equals $\co(f)$ for some $\hw$-morphism $f$.

Moreover, two object of $\tr{C}_{[0,1]}$ are isomorphic if and only if the corresponding two-term complexes are isomorphic in $\K(\hw)$ (that is, homotopy equivalent). %\K???

2. Let $(\tr{C}',w')$ be a weighted category, and assume that and additive functor $F:\hw\to \hw'$ is an equivalence. Then there exists a unique bijection $\tif$ between the class of isomorphic classes of objects of $\tr{C}_{[0,1]}$ and that of $\tr{C}'_{[0,1]}$ %that sends the isomorphism  class of $\co(s)\in \tr{C}_{[0,1]}$
 such that for any $\hw$-morphism $s$ we have $\tif([\co(s)])=[\co(F(s))]$.
\end{proposition}
\begin{proof}
1. The first part of the assertion %is %essentially a particular case of  \cite[Proposition 1.5.6(1)]{Bon10} 
 %easily follows from %Proposition 1.5.6(1) of~\cite{Bon10}  (cf. the proof of Corollary 1.5.7 of ibid, and take into account  Remark \ref{sic!w}(\ref{nota})). 
%Proposition 1.2.4(6,10) of \cite{Bon21}. 
given by %Proposition 1.2.4(10,12) of \cite{BoS19}.
Proposition 1.2.4(10) of \cite{BoS19}.  The second part follows from Theorems 3.2.2(II) and 3.3.1(VI), and Proposition 3.1.8(1) of ~\cite{Bon10} (cf.   Remark \ref{sic!w}(\ref{nota})). %ibid.

2. This statement easily follows from assertion 1. Indeed, %the %only possibility for
 the uniqueness follows from the fact that any element of  $\tr{C}_{[0,1]}$ is a cone of an $\hw$-morphism, and the existence follows from the isomorphism criterion provided by assertion 1.
\end{proof}

We will also need the following simple observations.

\begin{lemma}\label{ll168}
Assume that $(\cu,w)$ is a weighted category.

1. If %$(\cu,w)$ is a weighted category and 
 $\cu$ is idempotent complete then $\hw$ also is.

2. In particular, %this is the case
 both of these conditions are fulfilled if $\cu$ is smashing.

3. The class $\cu_{w=0}$ is connective (in $\cu$). %satisfies the following {\it connectivity} condition: $\cu_{w=0}\perp \cu_{w=0}[i]$ for any $i>0$.
\end{lemma}
\begin{proof}
1. Very easy; see Proposition 1.2.4.(7) of \cite{Bon22}.

2. Immediate from the previous assertion combined with Lemma \ref{l168}. 

3. Immediate from the orthogonality axiom (\ref{aort}) in Definition \ref{dws}. %(\ref{idmain}).
\end{proof}

\subsection{$t$-structures:  reminder}
\begin{definition}\label{dts}\hspace{1cm}
\begin{enumerate}
\item\label{itmain} A couple of subclasses  $\tr{C}_{t\leq 0}, \tr{C}_{t\geq 0}\subset \mathrm{Obj}\tr{C}$ is a 
 \textit{$t$-structure} on $\cu$ if the following assumptions are fulfilled:
\begin{enumerate}
\item $\tr{C}_{t\leq0}$ and $ \tr{C}_{t\geq0}$ are closed with respect to $\cu$-isomorphisms;
\item $\tr{C}_{t\leq0}\subset\tr{C}_{t\leq0}[1]$ and~${\tr{C}_{t\geq0}[1]\subset\tr{C}_{t\geq0};}$
\item $\tr{C}_{t\geq0}[1]\perp\tr{C}_{t\leq0};$
\item For any $M\in\mathrm{Obj}\tr{C}$ there exists a \textit{$t$-decomposition } distinguished triangle 

\begin{equation}\label{tdec} %{t-truncation}
L_tM\to M \to R_tM\to L_tM[1],
\end{equation}
where $L_tM\in\tr{C}_{t\geq0},\ R_tM\in\tr{C}_{t\leq0}[-1].$
\end{enumerate}

\item The full subcategory~$\Ht\subset\tr{C}$, whose object class equals ~${\tr{C}_{t=0}=\tr{C}_{t\leq0}\cap\tr{C}_{t\geq0},}$ is called the 
 \textit{heart} of $t$. 

\item Assume that $\tr{C}$ and $\tr{C}'$ are endowed with  $t$-structures $t$ and~$t'$
 respectively, $F:\tr{C}\to\tr{C}'$ is an exact functor. We say that  $F$ is \textit{$t$-exact} if ${F(\tr{C}_{t\leq0})\subset \tr{C}'_{t'\leq0}}$ and
 ~$F(\tr{C}_{t\geq0})\subset \tr{C}'_{t'\geq0}$.

\item Assume that $\tr{C}$ is endowed both with a weight structure  $w$  and a  $t$-structure $t$.
 We say that $w$ and $t$ are \textit{adjacent} if ${\tr{C}_{w\geq0}=\tr{C}_{t\geq0}.}$
%\end{itemize}
\end{enumerate}

\end{definition}

\begin{note}\label{t^can_w^st}\hspace{1cm}
\begin{enumerate}
\item\label{t^op} Clearly, the notion of $t$-structure is self-dual, that is, the couple   $t^{op}=(\tr{C}_{t\geq 0},\tr{C}_{t\leq 0})$ is a $t$-structure on the category $\tr{C}^{op}$.
\item\label{(Co)homol} In ~\cite{BBD82} where $t$-structures were introduced  and in several preceding papers of the first author the "cohomological convention" for $t$-structures was used. In the current text we use the so-called homological convention; in particular, we write  $\tr{C}_{t\geq0}$ instead of $\tr{C}^{t\leq0}$ (cf. Remark~\ref{sic!w}(\ref{nota})).
\item\label{w^st} For an abelian category $\A$, one can define the  \textit{canonical} $t$-structure on the derived category $D^*(\A)$
(where $*$ stands either for  $b,+,-,$ or $\varnothing$) as follows: take $D^*(\A)_{t^{can}\leq0}$
 (resp. $D^*(\A)_{t^{can}\geq0}$) to consist of those complexes whose cohomology in negative (resp. positive) degrees vanishes.
% При этом, $t$-раз\-ло\-жение комплекса $X^\bullet$ строится при помощи канонической фильтрации.
\end{enumerate}
\end{note}

Let us recall some basic properties of $t$-structures.

\begin{proposition}[{\cite[\S1.3]{BBD82}}]\label{sic!t}\hspace{1cm}
\begin{enumerate}
\item\label{L_r} %Треугольник~(\ref{t-truncation}) канонически и функториально определяется объектом~$M$.
The triangle (\ref{tdec}) is canonically (up to a unique isomorphism) and functorially determined by $M$. Moreover, $L_t$ is right adjoint to the embedding $\tr{C}_{t\geq 0}\to\tr{C}$. % (if we consider $ \cu_{t\ge 0}$ as a full subcategory of $\cu$). % and $R_t$ is left adjoint to  the embedding $ \cu_{t\le -1}\to \cu$ Треугольник, образующий $t$-разложение, канонически и функториально определяется объектом $M$.
% Более того, функтор $L_t$ сопряжён справа к вложению $\tr{C}_{t\geq 0}\to\tr{C}$.
\item\label{Ht=A} The heart $\Ht$  is %necessarily
  an abelian category with short exact sequences corresponding to distinguished triangles in  $\tr{C}$.
\item\label{H^t} Take $H_t$ to be the composition 
$$L_t\circ[1]\circ R_t\circ [-1].$$
 Then $H_t(\obj\tr{C})\subset\tr{C}_{t=0}$,  and $H_t$ is a  homological functor (if considered as a functor into $\Ht$) that %is
  restricts to a left adjoint to the embedding $\Ht\to \cu_{t\ge 0}$.
% identical on $\Ht$.
\item\label{tperp} $\cu_{t\le 0}=(\cu_{t\ge 0}[1])\perpp$. %1.3.4 of {BBD82}?!
\end{enumerate}
\end{proposition}

%The following statement 

\begin{proposition}\label{base_t-str}
Assume that $\tr{C}$ and $\tr{C}'$  are endowed with  $t$-structures~$t$ and $t'$, respectively, and $G:\tr{C}\to\tr{C}'$ is a $t$-exact functor.
\begin{enumerate}
\item\label{pi_*(Ht)} Then the restriction of  $G$ to $\Ht$ %is an 
 gives an exact functor %$\tr{HG}$ 
 $G^{\Ht}$ between the abelian categories $\Ht$  and $\Ht'$.

\item\label{tadj} Assume that there exists a left adjoint $G^*$ to $G$. Then the restriction of the functor $H_t\circ G^*$ to $\Ht'$ is left adjoint to %$\tr{HG}$.
 $G^{\Ht}$.

\item\label{pi_*(M)} Assume that $G$ is fully faithful. Then an object $M$ of $\tr{C}$ belongs to  $\tr{C}_{t\leq 0}$ (resp. $\tr{C}_{t\geq0}$, 
 $\tr{C}_{t=0}$) if and only if $G(M)$ belongs to $\tr{C}'_{t'\leq 0}$
 (соотв. $\tr{C}'_{t'\geq0}$, $\tr{C}'_{t'=0}$).
\end{enumerate}
\end{proposition}

\begin{proof}
\ref{pi_*(Ht)}. Obviously $G$  restricts to a functor $\Ht\to\Ht'$ indeed, 
 and Proposition~\ref{sic!t}(\ref{Ht=A}) easily implies that this functor is exact (as a functor between abelian categories).

\ref{tadj}. Immediate from Proposition 1.3.17(iii) of \cite{BBD82}.

\ref{pi_*(M)}. It clearly suffices to verify %the statement for the classes $\cu_{t\le 0}$ and $\cu_{t\ge 0}$
  that the preimages of  $\tr{C}'_{t'\leq0}$ and
 $\tr{C}'_{t'\geq 0}$ equal $\tr{C}_{t\leq 0}$ and $\tr{C}_{t\geq 0}$, respectively. Moreover, duality (see Remark~\ref{t^can_w^st}(\ref{t^op})) allows us to study the %$\cu_{t\ge 0}$-case of the statement only. 
 %можно рассмотривать только случай
 case   $G(M)\in\tr{C}'_{t'\geq0}$ only.

Since $G$ is $t$-exact, the triangle
$$
G(L_t M)\to G(M)\to G(R_tM)\to G(L_tM[1])
$$
is a $t'$-decomposition of the object $G(M)$. The uniqueness of  $t'$-decompositions
 (see Proposition~\ref{sic!t}(\ref{L_r})) implies $G(R_tM)=0$.
 Hence $R_tM$ is zero as well; thus  $M$ belongs to $\tr{C}_{t\geq0}$.
\end{proof}

\section{General statements on weight-exact localizations}\label{sgenwe}

\S\ref{sweadjt} contains a rich collection of (mostly) new and general results on weight-exact localizations and $t$-exact embeddings (with respect to adjacent $t$-structures).

In \S\ref{swelss} we prove some criteria for the weight-exactness of the localization $\cu\to \cu/\du$ in the case where $\du$ is semi-simple. This section is only needed for the construction of certain (non)-examples in \S\ref{scount} below.

\subsection{On weight-exact localizations and adjacent $t$-structures}\label{sweadjt}

We begin from recalling some properties of adjacent structures.

\begin{proposition}\label{Ht?Add}%\hspace{1cm}
Assume that  %$\cu$,  %(respectively, we can assume 
 $w$ is a weight structure on $\tr{C}$.

\begin{enumerate}
\item\label{iht1}
Assume that $\tr{C}$ is also endowed with a $t$-structure $t$ adjacent to $w$.    
\begin{enumerate}
\item\label{Ht<Add}  
Then the corresponding Yoneda-type functor $Y: \Ht\to\adfu(\hw^{op},\ab)$,
 $M\mapsto H_M=\tr{C}(-,M)$  %(see Definition \ref{dvtt}(3))
  is exact and fully faithful, and its essential image lies in  $\adfupr(\hw\opp,\ab)$ (see Definition \ref{Brown_Rep}(\ref{iadfupr}) and Remark \ref{rsmash}(2)).
	
\item\label{ProjHt} %If an adjacent $t$ exists then 
 $\Ht$ %^{\tr{E}}$ 
	has enough projectives, and the restriction $H_{t}^{\hw}$ of the functor
	%$t$-cohomology
 functor $H_t:\tr{C}\to \Ht$ provided by Proposition ~\ref{sic!t}(\ref{H^t})  %yields an equivalence of 
 %embeds  
 to $\hw$ is fully faithful and sends $\hw$ %${\Hw^{\tr{E}}}$ %эквивалентно подкатегории проективных объектов категории 
   into the subcategory of  projective objects of $\Ht$. %^{\tr{E}}$. 
	
	Moreover, each projective object $N$ of $\Ht$ is a direct summand of  $H_t(P)$ for some $P\in \cu_{w=0}$, and we can also choose $P\in \cu_{w=0}$ such that $N\cong H_t(P)$ whenever $\hw$ is idempotent complete. %adjunction isomorphism?!
	
	\item\label{i25} The bi-functor $\hw\opp\times \cu\to \ab,\ (M,N)\mapsto \cu(M,N),$ is isomorphic to the corresponding restriction of $\cu(-,H_t(-))$. 
	
	Consequently, %the aforementioned restriction of $H^t$ 
	 %the composition of $H^{t,\hw}$ with an equivalence 
	the composed functor $Y\circ H_t^{\hw}$  is isomorphic to the Yoneda-embedding $\hw\to \adfupr(\hw\opp,\ab)$. 
	
\end{enumerate}
\item\label{Ht=Add} Assume that $\tr{C}$  satisfies the Brown representability property (see Definition \ref{generated}(\ref{B_R}); in particular, $\tr{C}$ is smashing) and  $w$ is smashing. Then  a $t$-structure $t$ adjacent to $w$ exists and % the essential image of the aforementioned embedding equals the subcategory of $\adfu(\hw\opp,\ab)$ that consists of those functors that respect products (i.e., they send $\hw$-coproducts into products of abelian groups). 
  the aforementioned functor $\Ht\to\adfupr(\hw^{op},\ab)$ is an equivalence.
	
	%\item\label{Htp01} %Assume that a triangulated category $\tr{C}'$ i.
	 %Under the assumptions of 
\end{enumerate}
\end{proposition}

\begin{proof}
%\ref{iht1}????
 \ref{Ht<Add}. According to Theorem 4.4.2(4) of \cite{Bon10}, the Yoneda-type functor $\Ht\to\adfu(\hw^{op},\ab)$
is exact and fully faithful indeed. Hence it suffices to note that corepresentable functors respect products.

\ref{ProjHt}. Immediate from Proposition 5.2.1(1,2) of \cite{Bon24}.

\ref{i25}. %This statement 
 The bi-functor isomorphism in question  is given by the formula (25) that is %contained
 established in the proof of \cite[Theorem 4.4.2(3)]{Bon10}.

Restricting this isomorphism to $\hw\opp\times \hw$ one obtains the corresponding property of $Y\circ H_t^{\hw}$ immediately. 

\ref{Ht=Add}. This is (most of) Theorem 3.2.3 of~\cite{Bon24}.
\end{proof}

Now we pass to $t$-exact embeddings adjoint to weight-exact localizations.

\begin{theorem}\label{pi_&_G}
Assume that $\pi:\tr{C}\to\tr{E}$ is the localization of $\tr{C}$ by its triangulated subcategory~$\tr{D}$,
 and there exists a right adjoint  functor $\pi_*:\tr{E}\to\tr{C}$ to $\pi$.
\begin{enumerate}
\item\label{ImG} Then the functor $\pi_*$ is exact and fully faithful, and its essential image is %the subcategory of $\tr{C}$ whose object class equals 
 the class  $(\obj \tr{D})^{\perp}$.
\item\label{ImGw} Assume that  $\tr{C}$ and $\tr{E}$ are endowed with adjacent weight and $t$-structures 
 $(w,t)$ and $(w^{\tr{E}},t^{\tr{E}})$, respectively, and the functor  $\pi$ is weight-exact.
\begin{enumerate}
\item\label{G_t-exact} Then $\pi_*$ is $t$-exact and the essential image  $\pi_*(\Ht^{\tr{E}})$
 equals  $\cu_{t=0}\cap (\obj \tr{D})^{\perp}$.
\item\label{D=<U>} %Assume that
Assume that  $\tr{D}$ %=\langle \mathcal{U}\rangle $ for some class  $\mathcal{U}\subset\obj\tr{C}$ 
  equals the smallest strictly full triangulated subcategory of $\tr{C}$ that contains a class $\mathcal{U}\subset \obj \cu$ and is closed with respect to $\cu$-coproducts  (cf. Defintion \ref{generated}(\ref{localiz})). Then $\pi_*(\Ht^{\tr{E}})=\cu_{t=0}\cap(\cap_{i\in\Zbb}(\mathcal{U}[i]^{\perp}))$.
	
	\item\label{icomm}
	The diagram 
	\begin{equation}\label{ecomm}
	\begin{CD}
 \hw @>{H_t^{\hw}}>> \Ht\\
@VV{\pi^{\hw}}V@VV{\tr{H\pi}_*^*}V \\
\hw^{\tr{E}} @>{H_{t^{\tr{E}}}^{\hw^{\tr{E}}}}>>\Ht^{\tr{E}}
\end{CD}
\end{equation}
	is commutative up to an isomorphism; here $H_t^{\hw}$ and ${\pi^{\hw}}$ are the restrictions to $\hw$ (with the corresponding decreased targets) of the functors $H_t$ (see Proposition \ref{sic!t}(\ref{H^t})) and $\pi$, respectively, $H_{t^{\tr{E}}}^{\hw^{\tr{E}}}$ is the corresponding restriction of $H_{t^{\tr{E}}}$ to $\hw^{\tr{E}}$, and $\tr{H\pi}_*^*$ is the left adjoint to the restriction of $\pi_*$ to $\Ht^{\tr{E}}$ (see Proposition \ref{base_t-str}(\ref{pi_*(Ht)},\ref{tadj})).
	
	We will use the notation $Sq(\pi,w)$ for this diagram.

\item\label{u012} %Adopt the  assumptions of assertion \ref{ImGw}.\ref{D=<U>} and 
Assume  $\mathcal{U}=\mathcal{U}_0\cup \mathcal{U}_1\cup \mathcal{U}_2$, where 
  $\mathcal{U}_0$ consists of  cones of a class  $S_0$ of  $\hw$-morphisms,  the elements of
 $\mathcal{U}_1$ are  left $w$-degenerate, and the elements of $\mathcal{U}_2$ are right $w$-degenerate.

%\begin{enumerate}\item\label{U=u0u1u2}
 i. Then the class $\pi_*(\Ht^{\tr{E}})$ equals 
 $\cu_{t=0}\cap \mathcal{U}_0^\perp\cap\mathcal{U}_0[-1]^\perp$.
 Moreover, it consists of those $N\in \tr{C}_{t=0}$ such that the morphism $H_N(s)=\tr{C}(-,N)(s)$ is bijective for all $s\in S_0$.

%\item\label{H^t(N[i])}
ii. If $\mathcal{U}_1=\mathcal{U}_2=\varnothing$ then
 $M\in\obj\tr{C}$ belongs to $\pi_*(\tr{E})$ if and only if $H_t(M[i])\in \pi_*(\Ht^{\tr{E}})$ % (see Proposition~\ref{sic!t}(\ref{H^t}))
   for all $i\in\Zbb$.
\end{enumerate}	
\end{enumerate}
%\end{enumerate}
\end{theorem}
\begin{proof}
\ref{ImG}.  These statements are well known and contained in %\cite{neebook}. The exactness of $G$ is provided by Lemma 5.3.6 of ibid. %Verdier, Kapranov, Krause?!
\cite{Kra10}; see Corollary 2.4.2 and  Propositions 4.9.1 of ibid.  

%\ref{ImGw}.???
 \ref{G_t-exact}. $\pi_*$ is $t$-exact according to  Proposition 4.4.5(3) of \cite{Bon10}.
 %В силу предложения~\ref{base_t-str}(\ref{pi_*(M)}), класс объектов $G(\tr{E})$  замкнут относительно изоморфизмов. 
Next, Proposition~\ref{base_t-str}(\ref{pi_*(M)}) implies that $\pi_*(\Ht^{\tr{E}})$ is the intersection of the essential image of $\pi_*$ with  $\cu_{t=0}$. Hence it remains to apply  assertion \ref{ImG}.
%applying assertion \ref{ImG}  combined with %Proposition \ref{prtst}(\ref{itex})
 %~\ref{base_t-str}(\ref{pi_*(M)}) we obtain the equality for $G(\eu_{t_{\eu}=0})$ in question.

\ref{D=<U>}. Since for any $N\in \obj \tr{C}$ the class of those $M\in \obj \tr{C}$ such that %$M\perp (\cup_{i\in \z} \{N[i]\})$ 
 $M\perp \{N[i]:\ i\in \z\}$ is a triangulated subcategory of $ \tr{C}$ closed with respect to  $\tr{C}$-coproducts, it remains to apply the previous assertion.

\ref{icomm}. Since $\cu_{w=0}\subset \cu_{w\ge 0}=\cu_{t\ge 0}$ and $\pi$ is weight-exact, it suffices to verify that the corresponding "extension"
of $Sq(\pi,w)$ 
	$$\begin{CD}
 \cu_{t\ge 0} @>{H_t^{\cu_{t\ge 0}}}>> \Ht\\
@VV{\pi^{\cu_{t\ge 0}}}V@VV{\tr{H\pi}_*^*}V \\
\tr{E}_{t^{\tr{E}}\ge 0} @>{H_{t^{\tr{E}}}^{\tr{E}_{t^{\tr{E}}\ge 0}}}>>\Ht^{\tr{E}}
\end{CD}$$ %\tr{E}^{\tr{E}^{t^\tr{E}}\ge 0}
is commutative. Now, recall (see Proposition \ref{sic!t}(\ref{H^t})) that the functors in this diagram are left adjoint to the functors in the square
$$\begin{CD} \Ht^{\tr{E}}@>{}>>\tr{E}_{t^{\tr{E}}\ge 0}\\
@VV{\pi_*^{\Ht^{\tr{E}}}}V@VV{\pi_*^{\tr{E}_{t^{\tr{E}}\ge 0}}}V \\
\Ht @>{}>> \cu_{t\ge 0} \end{CD}$$
and the latter square commutes (up to an isomorphism) since $\pi_*$ is $t$-exact (see %Proposition
 assertion %\ref{ImGw}.
  \ref{G_t-exact}).

%\ref{u012}. 
  \ref{u012}(i). Since $w$ and $t$ are adjacent, the orthogonality axioms in Definitions \ref{dws} and ~\ref{dts}(\ref{itmain}) imply that
\begin{equation}\label{eort}
\tr{C}_{w\leq-1}\perp\tr{C}_{w\geq0}=\tr{C}_{t\geq0},\enskip \tr{C}_{w\geq1}=\tr{C}_{t\geq0}[1]\perp\tr{C}_{t\leq0}.
\end{equation}
Hence %$\tr{C}_{w=j}\perp\tr{C}_{t=0}$ for $j\neq0$, 
 $\mathcal{U}_1[i]\perp\tr{C}_{t=0}$ and
 $\mathcal{U}_2[i]\perp\tr{C}_{t=0}$ for all $i\in\Zbb$.
It remains to study $(\cup_{i\in \z} \mathcal{U}_0[i]) \perpp$.

%Take $s\in \hw(A,B)$ and $P=\co(s)$.  %then we obtain %$\co(s)
 %$P[i]\perp \cu_{t=0}$ for any $j\neq -1,0$. Moreover, 
Take $s:A\to B\in S_0$, where $A,B\in\tr{C}_{w=0}$, and $U=\mathrm{Cone(s)}\in\mathcal{U}_0$.
The %equalities 
 orthogonalities (\ref{eort}) imply that 
  $\tr{C}_{w=j}\perp  \tr{C}_{t=0}$ for any $j\neq 0$.
 Hence for any $N\in \cu_{t=0}$ we have an exact sequence %$$\ns\to  \cu(U[-1],N)\to \cu(B,N)\to \cu(A,N)\to  \cu(U,N)\to \ns. $$ 
$$
0\to\tr{C}(U[-1],N)\to\tr{C}(B,N)\xrar{H_N(s)}\tr{C}(A,N)\to\tr{C}(U,N)\to0.
$$
Hence $N\in U^\perp\cap U[-1]^\perp$ if and only if $H_N(s)$ is invertible.
Similarly we obtain $U[i]\perp \cu_{t=0}$ for any $i\neq -1,0$; this fact also follows from (\ref{eort}) combined with Proposition \ref{p01}(1). Thus $\{ U[i]:\ i\in \z \} \perpp\cap \tr{C}_{t=0}= U^\perp\cap U[-1]^\perp \cap \tr{C}_{t=0}$.
%$U\perpp\cap  \cu_{t=0}$  consists of all those $N\in \cu_{t=0}$ such that $\cu(s,N)$ is invertible. 
 
 Combining all these observations we immediately obtain the equalities for $\pi_*(\tr{E}_{t_{\tr{E}}=0})$  in question.
  
%\ref{H^t(N[i])}
 \ref{u012}(ii). If $M\cong \pi_*(Y)$ for some $Y\in\tr{E}$, then the $t$-exactness of $\pi_*$ %along with Proposition \\ref{sic!t}(\ref{L_r}){H^t} yields
  yields (see Proposition \ref{sic!t}(\ref{L_r},\ref{H^t})) that  $H_t(M[i])\cong \pi_*(H_t(Y[i]))\in \pi_*(\tr{Ht}^{\tr{E}})$. % в силу $t$-точности  функтора $G$ и предложения~\ref{sic!t}(\ref{L_r}).

To prove the inverse  implication for any %Докажем в обратную сторону. Для каждого 
$U=\co(s)\in\mathcal{U}_0$ as above  we consider the long exact sequence
$$
\ldots\to\tr{C}(U[-1],M)\to\tr{C}(A,M)\to\tr{C}(B,M)\to\tr{C}(U,M)\to\ldots.
$$
%According to 
Applying Proposition \ref{Ht?Add}(\ref{i25}) we obtain that for any~$i\in\Zbb$ the homomorphism $\tr{C}(s,M[i])$ is isomorphic (in the category of $\ab$-arrows) to $\tr{C}(s,H_t(M[i]))$; hence it is bijective. 
Thus  $U[i]\perp M$ for all $i\in\Zbb$ and
 $U\in\mathcal{U}_0$. According to assertion %~\ref{ImGw}.
  \ref{D=<U>}, the essential image class $\pi_*(\tr{E})$ equals 
 $\cap_{i\in\Zbb}(\mathcal{U}_0[i]^\perp)$; hence it contains $M$.
\end{proof}

\begin{note}\hspace{1cm}
\begin{enumerate}
\item In parts %~\ref{ImGw}.
\ref{D=<U>} %и \ref{U=u0u1u2} 
and ~\ref{u012} of our theorem we do not assume ~\tr{C} to be smashing.
 %приведена. Рассматриваются только те копроизведения, которые существуют в~$\tr{C}$.
\item  The proof will not really change if we will assume that  $\tr{D}$  equals the smallest strictly full \textbf{retraction-closed} triangulated subcategory of $\tr{C}$ that contains a class $\mathcal{U}\subset \obj \cu$ and is closed with respect to $\cu$-coproducts (yet cf. Lemma \ref{l168}).
%Note however that this additional assumption is superfluous  if $\tr{C}$ is closed with respect to (at least) countable coproducts; see %Proposition 1.6.8 of \cite{Nee01}.
\item One can certainly take %~\ref{D=<U>} можно взять 
$\mathcal{U}=\obj\tr{D}$. Moreover, for any  $U\in\mathcal{U}$ and $N\in\Ht$ one can compute $\tr{C}(U,N)$
 in terms of the \textit{weight complex} of $U$ (see  Theorem 4.4.2(9) of \cite{Bon10} and Theorem 2.1.2(1,2) of \cite{Bon21}). % and Proposition 2.2.2(10) of \cite{Bon24}).
 
On the other hand, taking $\mathcal{U}=\mathcal{U}_0\cup \mathcal{U}_1\cup \mathcal{U}_2$ appears to be the main general way of constructing weight-exact localizations (cf. Remark 3.1.4(3) of  \cite{BoS19}).
\end{enumerate}
\end{note}
Most of notions mentioned in the following theorem %refers to 
 can be found in Definition \ref{generated}.
\begin{theorem}\label{main}
Assume that  $\tr{C}$ is  well generated %(see Definition \ref{generated}(\ref{iwell})) 
 and $\tr{D}$ is its localizing subcategory  generated by a set $\mathcal{U}\subset\obj\tr{C}$.
\begin{enumerate}
\item\label{Brown_exists} Then both $\tr{C}$ and the Verdier localization ~\tr{E}=\tr{C}/\tr{D} are locally small %discuss in ls remark??!
and satisfy the Brown representability condition. % (see Definition \ref{generated}(\ref)).
 Moreover,  the localization functor  $\pi: \tr{C}\to\tr{E}$ respects coproducts and possesses an exact right adjoint  $\pi_*$.
\item\label{S_0} Assume that $w$ is a smashing weight structure on $\tr{C}$, and a {\bf set} $\mathcal{U}\subset\obj\tr{C}$ equals
  $\mathcal{U}_0\cup\mathcal{U}_1\cup\mathcal{U}_2$, where $\mathcal{U}_0$ consists of cones of a set $S_0$ of $\hw$-morphisms, the elements of $\mathcal{U}_1$ are left $w$-degenerate, and the  elements of   $\mathcal{U}_2$ are right $w$-degenerate. %Тогда
\begin{enumerate}
\item\label{www}  Then there exists a smashing weight structure ${w^{\tr{E}}}$ on  $\tr{E}$ such that $\pi$
 is weight-exact (that is, $w$ descends to $\tr{E}$). 

\item\label{t,t_E} There exist   $t$-structures $t$ and $t^{\tr{E}}$ on $\cu$ and  $\tr{E}$, respectively, that %is 
 are adjacent to  $w$ and $w^{\tr{E}}.$

\item\label{pi_*(Ht_E)}  $\pi_*$ restricts to an equivalence of $\Ht^{\tr{E}}$ with the full exact subcategory of %эквивалентно категории аддитивных функторов~${\Hw^{\tr{E}}^{op}\to\Ab,}$ сохраняющих произведения.
  $\Ht$ that is characterized by any of the following conditions.
\begin{enumerate}
\item $N\in\obj\Ht: (\mathcal{U}_0\cup\mathcal{U}_0[-1])\perp N$;
\item $N\in\obj\Ht: H_N(s)=\tr{C}(s,N)$ is %invertible обратим для всех 
 bijective for all $s\in S_0$.
\end{enumerate}
%\item\label{ProjHt_E} $\Ht^{\tr{E}}$ has enough projectives, and the  functor $H^t:\tr{C}\to \hrt$ provided by Proposition ~\ref{sic!t}(\ref{H^t})  yields an equivalence of  ${\Hw^{\tr{E}}}$    with the subcategory of  projective objects of $\Ht^{\tr{E}}$. 
\end{enumerate}
\end{enumerate}
\end{theorem}
\begin{proof}
\ref{Brown_exists}. Immediate from Proposition  3.5.1 and   Theorems 5.1.1 and 7.2.1(2) of \cite{Kra10}.

\ref{www}. This is  Theorem 3.1.3(3(ii)) of ~\cite{BoS19}. 

\ref{t,t_E}. By Proposition \ref{Ht?Add}(\ref{Ht=Add}), the existence of  $t$ and $t^{\tr{E}}$   follows from the Brown respresentability properties for $\tr{C}$ and $\tr{E}$. % according to Существование $t$ и $t_E$ следует из свойства представимости Брауна  и предложения
 
\ref{pi_*(Ht_E)}. This is an obvious consequence of Proposition \ref{base_t-str}(\ref{pi_*(M)})
 combined with Theorem \ref{pi_&_G}(\ref{u012}.i).%{U=u0u1u2}).

%\ref{ProjHt_E}. %The category $\tr{C}$ is closed with respect to countable coproducts; hence 
 %By Lemma \ref{l168}, % any idempotent in $\cu$ splits (see Proposition 1.6.8 of  \cite{Nee01}).
 %$\cu$ is idempotent complete. Consequently, the statement follows from Proposition 5.2.1(1--2) of \cite{Bon24}.
\end{proof}

\begin{note}
So, we give a rather indirect description of $\hw^{\eu}$ in terms of $\hw$ and $S_0$. 

Though it is fine for our purposes, it is worth noting that the restricted functor $%\underline{H\pi}
\pi^{\hw}: \hw\to \hw^{\eu}$ can also be described as the composition of the additive localization functor $\hw\to {\hw[\mathcal{S}^{-1}],}$ with the idempotent completion one. Here one takes  $\mathcal{S}$ to be the closure of ${S_0\cup \{\id_M, M\in \cu_{w=0}\}}$ with respect to coproducts and applies Theorem 3.2.2(4) and Remark 3.2.3(1) of \cite{BoS19}. Moreover, one can describe morphisms in $\hw[\mathcal{S}^{-1}] $ via certain explicit formulae closely related to non-commutative localizations that we mentioned in \S\ref{shist}; see Remark 3.2.3(2) of ibid. and Proposition %s 3.2.2.2 and 
 3.1.7 of \cite{BoS18}. 
\end{note}

\subsection{Weight-exactness of localizations by semi-simple subcategories}\label{swelss}

\begin{proposition}\label{C'=C/D}
 Assume that $(\cu,w)$ is a weighted category and $\tr{D}\subset \cu$ is a semi-simple\footnote{That is, any morphism in $\tr{D}$ %factors as 
 can be presented $ id_{M_0} \bigoplus 0:M_1\to M_2$ for some $M_i\in \obj\tr{D}$. %\footnote{It is well known 
 Recall that a triangulated category is semi-simple if and only if it is abelian; cf.   Proposition II.1.2.9 of \cite{Ver96}.} 
triangulated subcategory such that the Verdier localization 
 ${\tr{E}=\tr{C}/\tr{D}}$ is locally small.
\begin{enumerate}
\item\label{desc} Then $w$ descends to $\tr{E}$  if and only if any $N\in \obj \du$ can be presented as $N_-\bigoplus N_+$, where $N_-\in \cu_{w\le 0}$ and $N_+\in \cu_{w\ge 0}$. 
\item\label{primes} Assume that $w$ is smashing and any object of $\tr{D}$ is isomorphic to a coproduct of indecomposable objects. %fulfilled iff???!!
 Then the equivalent assumptions of the previous assertion are fulfilled if and only if any indecomposable object of $\tr{D}$ % belongs either to $\cap_{i\in \z}\cu_{w\le i}$ 
 is either left or right $w$-degenerate or belongs to  $\cu_{[s,s+1]}$ for some  $s\in\Zbb$.
\end{enumerate}
\end{proposition}

\begin{proof}
\ref{desc}. According to Theorem 3.1.3(1) of \cite{BoS19}, $w$ descends to $\tr{E}$ if and only if
	  for any $N\in \obj \tr{D}$ there exists a $\tr{D}$-distinguished triangle
 \begin{equation}\label{LN->N->RN}
LN\xrar{i}N\xrar{j}RN\to LN[1],
\end{equation}
such that $i$ factors through (an element of) $\cu_{w\le 0}$ and $j$ factors through $\cu_{w\ge 0}$. If $N= N_-\bigoplus N_+$, where $N_-\in \cu_{w\le 0}$ and $N_+\in \cu_{w\ge 0}$ then we can take $i$ (resp. $j$) %=\id_{N_-}\bigoplus 0$ and $j=\id_{N_+}\bigoplus 0$.
 to be the obvious embedding $N_-\to N_-\bigoplus N_+$ (resp. the projection $N_-\bigoplus N_+\to N_+$).

Now let us verify the converse implication. Assume that we have a distinguished triangle (\ref{LN->N->RN}). Using semi-simplicity, present $i$ as %$\id_{N_-}\bigoplus (0:N'\to N_+)$ 
$
 \begin{pmatrix}
 0 & id_{N_{-}}\\
 0 & 0
 \end{pmatrix}: N'\oplus N_{-}\to N_{-}\oplus N_{+}$ for $N_-,N_+,N'\in \obj \tr{D}$; then $j$ can be presented as %$\id_{N_+} \bigoplus (0:N_-\to N'[1])$. 
$
 \begin{pmatrix}
 0 & id_{N_{+}}\\
 0 & 0
 \end{pmatrix}: N_{-}\oplus N_{+}\to N_{+}\oplus N'[1]$. By our assumptions, $i$ factors through $\cu_{w\le 0}$, $\id_{N_-}$  factors through $\cu_{w\le 0}$ as well; hence $N_-\in \cu_{w\le 0}$. Similarly we obtain $N_+\in \cu_{w\ge 0}$.

\medskip
\ref{primes}. Assume that the assumptions of assertion~\ref{desc} are fulfilled an  object $N$ of $\tr{D}$ is indecomposable; then $N[-i]$ is indecomposable %as well.
 for any $i\in\Zbb$. Consequently, for any $i$ we either have
  ${N[-i]=N[-i]_{-}\in\tr{C}_{w\leq0}}$ or $N[-i]=N[-i]_{+}\in\tr{C}_{w\geq0}$.
 If the first equality is fulfilled for all $i\in\Zbb$ then  $N\in\cap_{i\in\Zbb}\tr{C}_{w\leq i}$, that is, $N$
 right $w$-degenerate. If $N[-i]=N[-i]_{+}$ for all $i\in\Zbb$ then $N$ is  left $w$-degenerate. Lastly, if neither of the equalities is fulfilled for all $i\in\z$ then there exists
  $s\in\z$ such that $N\in \tr{C}_{w\leq s+1}$ и $N\notin \tr{C}_{w\leq s}$. Thus $N\in\tr{C}_{[s,s+1]}$.

Conversely, assume that  $N=\coprod N_j$, where $N_j$ are indecomposable objects of $\tr{D}$ that are either left or right $w$-degenerate or belong to  $\cu_{[s,s+1]}$ for some $s\in \z$. Then we can take  $N_{+}$ to be the coproduct of those $N_j$ that belong to  $\tr{C}_{w\geq0}$, and  $N_{-}$ to be the coproduct of the remaining ones. 

Alternatively, one can apply Theorem 3.1.3(3) of \cite{BoS19} here directly.
\end{proof}

\section{On weight structures with projective hearts}\label{Loc_D(R)}

In this section we mostly study weight structures with the heart equivalent to $\projr$ (for a ring $R$).

We start  \S\ref{w^PD(A)} with the simple Proposition \ref{hprojr} on the weight structure corresponding to a single compact generator that %does not admit non-trivial self-extensions of positive degrees
 satisfies the connectivity condition (see Definition \ref{dwso}(\ref{iconn})). %(cf. Lemma \ref{ll168}(3)). 
 We also study the so-called projective weight structure $w^{\pr}$ on the derived category $D(\rmod)$. Below we will often mention the relation of $w^{\pr}$ to the projective dimension of complexes; see Proposition \ref{w^PP}(\ref{pd=wd}).

In \S\ref{Bazz} we %restrict
 apply  the results of section \ref{sweadjt} to the case $w=w^{\pr}$ and $\uu=\uu_0$ consisting either of (i) $U$, where $u:R\to U$ is a  ring epimorphism satisfying certain assumptions, or of  (ii) $\{R[s^{-1}]\}$, where $s$ runs through elements of an ideal $I$ or $R$ and $R$ is assumed to be commutative. In these cases the corresponding essential images $\Ht^{\eu}\to \Ht$ were called %{\it contramodules} 
 $u$-contramodules and $I$-contrmodules in several papers of L. Positselsky and his co-authors; cf. also Remark \ref{rpos}(\ref{idwg}).

In \S\ref{scomphearts} we prove that "the heart information" on a weight-exact localization %satisfying ("our usual") 
 constructed using Theorem \ref{main}(\ref{S_0}) %.\ref{t,t_E}) 
 is canonically determined both by $(\hw,S_0)$  (see Theorem \ref{tp01ext}(2)) and by $(\Ht,\pi_*(\Ht^{\eu}))$ (see Remark \ref{rdetht}). Next (in Corollary \ref{contra}) we apply this observation to well generated weighted categories with $\hw\cong \projr$.

\subsection{The projective weight structure on $D(\rmod)$}\label{w^PD(A)}
%$D(\rmod)$?????????!!

We start from a rather general (and yet simple) statement on the existence of weight structures. Starting from this moment, $R$ will always denote an associative unital ring, %A "module" will  always mean an left %$R$-
 %module if will not say that this module is a right one explicitly. 
 all unspecified modules are left modules,  $\rmod$ is the category of (left) $R$-modules, and  $\projr$ is %the category of projective (left) $R$-modules.
 its subcategory of projective %(left)
  $R$-modules.

\begin{proposition}\label{hprojr}
 Assume that $\tr{C}$ is compactly generated by a connective object class $\PP=\{P\}$ (for some $P\in \obj \cu$) %and $P\perp P[i]$ for all $i>0$ (cf. Lemma \ref{ll168}(3))
  (see Definition \ref{dwso}(\ref{iconn})); set $R=\ndo(\PP)^{op}$ %to be 
	 (the ring opposite to the endomorphism ring of $P$). %respectively, there is an equivalence 
 %and additive functor $F:\projr\to \hw'$ is an equivalence,

1. Then $\tr{C}$ is well generated, the set $\PP$  class-generates a smashing weight structure $w$ with $\hw\cong \projr$, and the elements of $\cu_{w=0}$ are the direct summands of all coproduct powers of $P$. %; a quasi-inverse to this equivalence sends $R$ into $P$.

2. %Consequently,
There exists a $t$-structure $t$ adjacent to $w$ and $\Ht\cong \rmod$. Moreover, $\cu_{t\ge 0}=\cu_{w\ge 0}=\{P[i]:\ i<0\}\perpp$ and $\cu_{t\le 0}=%\cu_{t\le 0}=
 \{P[i]:\ i>0\}\perpp$.
\end{proposition}
\begin{proof}
1. Clearly, $\cu$ is well generated; see Remark \ref{rwell}(1). %The remaining statements easily follow from %This is an easy consequence of  \cite[Theorem 3.2.3(3)]{Bon21}
%Theorem 3.2.3(3,6) of \cite{Bon21};  cf. Remark 2.2.4(2) of ibid. and the proof of assertion 2. %see also Proposition 4.3.1.3.(III.3) of \cite{BoS18}.
 Then the existence of $w$ and the description of  $\cu_{w=0}$ is given by Theorem 3.2.3(3) of \cite{Bon21} and part 6 of loc. cit. says that $\hw$ is equivalent to the subcategory of projective objects in $\Ht$ (cf. Proposition 4.3.3(2) of ibid.). Thus it remains to compute    $\Ht$, and we will do this in the proof of assertion 2.

2. $t$ that is adjacent to $w$ exists according to Theorem \ref{main}(%\ref{S_0}.
\ref{t,t_E}) (one can take $\tr{D}=\ns$ in this theorem).
%Next, $\Ht$ is equivalent to the category $ \adfupr(\hw^{op},\ab)$; see Proposition \ref{Ht?Add}(\ref{Ht=Add}). 
Next, $\Ht$ is equivalent to the category of those contravariant functors  from the one-object category $\{P\}$ into $\ab$ that respect the addition of morphisms; see Theorem 3.2.3(6) of \cite{Bon21}. The latter category is obviously equivalent to $\rmod$. %left, right???!!!

Moreover, Proposition \ref{sic!t}(\ref{tperp}) implies  $\cu_{t\le 0}=(\cu_{t\ge 0}[1])\perpp=\cu_{w\ge 1}\perpp$, and our description of $w$ easily implies that this class equals  $\{P[i]:\ i>0\}\perpp$. Lastly, $\cu_{t\ge 0}=\cu_{w\ge 0}$ since $t$ is adjacent to $w$; hence it remains to apply Theorem 3.2.3(6) of \cite{Bon21}. %[BoS19, Cor.2.3.1(1)]???
\end{proof}

Now we pass to the case $\tr{C}=D(\rmod)$.

\begin{proposition}\label{w^PP}\hspace{1cm}
\begin{enumerate}
\item\label{drgen} The category $\tr{C}=D(\rmod)$ is locally small and compactly generated by the set $\PP=\{P\}.$

\item\label{class-wPP}  $\tr{C}$ is endowed with a weight structure $w^\pr$  class-generated by $\PP$. % the set  $\PP=\{P\}.$ 
Moreover, $w^{\pr}$ is non-degenerate and adjacent to the $t$-structure $t^{can}$ (see Remark \ref{t^can_w^st}(\ref{w^st})). Furthermore, the category  $\hw^\pr$ lies in $ \Ht^{can}$ and consists of complexes whose (only) zeroth cohomology is projective. 

%\item\label{degen}  $w^\pr$ is non-degenerate. %Композиция $$D(\A)\cong \K(\proj)_{hpr}\hookrightarrow \K(\A)$$ весо-точна  относительно структур $w^{\Pr}$ и~$w^{st}$.
%\item\label{w-adj-t} Весовая структура $w^{\pr}$ и $t$-структура $t^{can}$ на $D(\A)$ смежны.
\item\label{pd=wd} Let $i,j\in \z$. 
Then  $M\in %D(\rmod)
\tr{C}_{w^\pr\geq i}$ belongs to 
 $%D(\rmod)
\tr{C}_{w^\pr\leq j}$ if and only if ${\mathrm{pd}_{R} M\leq j}$, that is, there exists a bounded $\projr$-complex $N=N^\bullet$ whose terms $N^s$ vanish for $s<-j$ and $N\cong M$ (in $D(\rmod)$).
\end{enumerate}
\end{proposition}

\begin{proof}

\ref{drgen}. %$\tr{C}=D(\rmod)$ is  well known to be  locally small; see \S4.14 of \cite{Krause}. %locally small by the well known  Proposition 10.4.4 of \cite{Wei94}. %Theorem 10.4.8: resolutions?! Proposition 2.4.4(b,c1) of {Ver96}: better?!
%The fact that $R\in\tr{C}$ is compact and generates $D(\rmod)$ (as its own localizing subcategory) 
$\tr{C}=D(\rmod)$ is  well known to be  locally small and compactly generated by $R$; see \S5.8 of \cite{Kra10}.

\ref{class-wPP}. The existence of $w^\pr$ class-generated by $\PP$ follows from Proposition \ref{hprojr}(1).  $w^\pr$ is non-degenerate by Theorem of \cite{Bon22}.  Proposition \ref{hprojr}(2) gives the existence of a $t$-structure $t$ adjacent to $w^\pr$. Moreover,  we have $\cu(R[i],N)=H^{-i}(N)$ for a complex $N\in \obj \cu$; hence the description of $t$ given by Proposition \ref{hprojr}(2)  implies $t=t^{can}$.

Lastly, recall that $\hw^\pr$ consists of the direct summands of all coproduct powers of $P$; this finishes the proof.

\ref{pd=wd}. Recall that $\cu_{[i,j]}$ equals the smallest subcategory of $\cu$ that is closed with respect to extensions and contains $\cup_{m\le j\le n}\cu_{w^\pr=j}$; see Proposition 1.3.3(3) and the proof of Corollary 1.5.7  of \cite{Bon10}.

Now, if $M\cong N^{\bullet}$ with $N^s=0$ for $s<-j$ and $s\gg 0$ and $N^s\in \obj \projr$ for all $s\in\z$ then $M$ belongs to the smallest subcategory of $\cu$ that is closed with respect to extensions and contains $N^s[-s]$; hence $M\in \cu_{w^\pr\le j}$.

Let us verify the converse implication. %Next, since %
 Since $M$ belongs to %$M\in %D(\rmod)
$\cu_{w^\pr\geq i}%D(\rmod)
=\cu_{t^{can}\geq i}$, it also belongs to the essential image of the well known  %full exact embedding
 exact fully faithful functor  $K^-(\projr)\to D(\rmod)$; see Proposition IIII.2.4.4(b,c1) of \cite{Ver96}. Hence if $M\in \cu_{[i,j]}$ then it belongs to the essential image (with respect to this embedding) of  the smallest subcategory of $K^-(\projr)$ that is closed with respect to extensions and contains $\cup_{i\le s\le j}\projr[-s]$. It clearly follows that ${\mathrm{pd}_{R} M\leq j}$.
\end{proof}

%\begin{corollary}\label{Ht^can} ${\A\cong\Ht^{can}\cong\adfupr(\proj^{op},\ab)}$. \end{corollary}

%\begin{proof}
%Хорошо известно, что функтор вычисления нулевых когомологий задаёт эквивалентность $H^0:\Ht^{can}\to\A$. Категория $D(\A)$ хорошо порождена,
 %prove?? и в ней выполняется свойство представимости Брауна. 
%Применяя предложения~\ref{w^PP}(\ref{class-wPP}),(\ref{w-adj-t})  и~\ref{Ht?Add}(\ref{Ht=Add}), получаем эквивалентность $\Ht^{can}\cong\adfupr(\proj^{op},\ab)$. \end{proof}
\begin{note}\label{rdwy}
Before passing to examples corresponding to contramodules, let us discuss the relation of our results to \cite{Dwy06}. 

Recall that  the main subject of  %\cite{Dwy06}
 ibid. was the (essential image via $\pi_*$ of the) localization of $\cu=D(\rmod)$ by $\du=\langle \co(S_0)\rangle^{cl}_{\cu}$, where $S_0$ is a set of $\fprojr$-morphisms.

Applying Proposition \ref{w^PP} %we obtain that  Proposition 3.1 of \cite{Dwy06} easily follows
 one can easily deduce Proposition 3.1 %of \cite{Dwy06}  
  ibid. from Theorem \ref{thea}(\ref{wwwa}) (along with Lemma \ref{ll168}(3)).   %essentially gives the connectivity %(see Lemma \ref{ll168}(3))
 %of the corresponding $\pi_*(R)$; thus it % \ref{main}(\ref{S_0}.\ref{www}). 
%; recall that the main subject of ibid.  Now, connectivity is closely related to (hearts of) weight structures; see  Lemma \ref{ll168}(3) and Proposition \ref{hproj}(1). Next
 %Moreover, one can easily Alternatively,
  Moreover,  Proposition 3.4 of ibid. is easily seen to follow from our  Theorem \ref{thea}(\ref{iessi}).
\end{note}
\subsection{Contramodulizations of derived categories}\label{Bazz}

In this subsection we will also assume $\tr{C}=D(\rmod)$. %, where $R$ is an associative unital ring. 
 Our goal is to describe certain examples to %Theorem \ref{main} (and Theorem \ref{pi_&_G}) 
 Theorem \ref{thea} that can be found in literature. To this end, we 
%Now we will define
  define two types of sets $\mathcal{U}\in \tr{C}_{[0,1]}$ that correspond to %certain %\it??
 so-called contramodules.

%Напомним следующие хорошо известные определения.

\begin{definition}[{\cite[Definition 4.5]{GeL91}}, {\cite[\S1]{BaP20}}]\label{ucontra}
Let $u:R\to U$ be a ring homomorphism.

1. %A homomorphism of associative rings  $u:R\to U$
 $u$  is said to be a \textit{homological (ring) epimorphism} %[GeL91]??!!!
if the obvious homomorphism ${U\otimes_R U\to U}$ is bijective and ${\mathrm{Tor}^R_i(U,U)=0}$ for all ${i>0.}$

2. An $R$-module %left: above
is said to be a \textit{(left) $u$-contramodule} if $\rmod(U,C)=\{0\}=\mathrm{Ext}^1_R(U,C).$

We will write $\rmod_{u-ctra}$ for the (full) subcategory of $\rmod$ that consists of  $u$-contramodules.
\end{definition}

%Следуя~\cite{Pos17}, введём следующие обозначения.
\begin{definition}
%Пусть $R$\,---\,коммутативное кольцо.
 Assume that $R$ is commutative and $C$ is an $R$-modules.
\begin{enumerate}
\item %Для фиксированного $s\in R$ рассмотрим мультипликативную систему
 %$\{1,s,s^2,\ldots,s^n,\ldots\}\subset R$. $R$-модуль $C$ называется
 %\textit{$s$-контрамодулем}, если $\rmod(R[s^{-1}],C)=0=\mathrm{Ext}^1_R(R[s^{-1}],C)$.
 For $s\in R$ we will say that %an $R$-module 
$C$ is an \textit{$s$-contramodule} if  $\rmod(R[s^{-1}],C)=\ns=\mathrm{Ext}^1_R(R[s^{-1}],C)$.\footnote{Cf. Lemma 2.1 of \cite{Pos17} for an nice re-formulation of this condition.}
\item Assume that  $I$ is an ideal of $R$. We %$R$-модуль $C$ называется  \textit{$I$-контрамодулем}, если $C$~--- $s$-контрамодуль для всех $s\in I$.
 say that $C$ is an \textit{$I$-contramodule} if  $C$ is an $s$-contramodule for any $s\in I$.

We will write $\rmod_{I-ctra}$ for the corresponding subcategory of $\rmod$.
\end{enumerate}
\end{definition}

\begin{note}
According to Theorem 5.1 of \cite{Pos17}, if $I=\sum_{j\in J}Rs_j$ (for some $s_j\in I$) then $C$ is an $I$-contramodule %whenever
 if and only if it is an $s_j$-contramodule for all  $j\in J$. 
%summation: f. ok?? 
\end{note}

Now we put these two type of contramodules in a single statement.

%Докажем следующую теорему (ср. \cite[Th.6.1(b)]{BaP20}).
\begin{theorem}\label{tloc}
Assume that either (i) $u: R\to U$ is a homological (ring) epimorphism  and ${\mathrm{pd}_R U\leq 1}$, or (ii) $R$ is commutative and $I$ is an ideal in it. 
 
In %the first 
case (i) set $\mathcal{U}=\{U\}$ and  $\uu\con=\rmod_{u-ctra}$, %to be the subcategory of  and %in the second case 
 and in case (ii) set $\mathcal{U}=\{R[s^{-1}]:\ s\in I\}$ and $\uu\con=\rmod_{I-ctra}$. Take $D=\langle \mathcal{U}  \rangle^{cl}_{\tr{C}}$.
\begin{enumerate}
\item\label{w_trE}  Then the Verdier localization $\tr{E}=\tr{C}/\tr{D}$ is  locally small and the localization functor  $\pi: \tr{C}\to\tr{E}$ %respects coproducts and 
 possesses an exact right adjoint  $\pi_*$.
 
\item\label{ip01} We have $\mathcal{U}\subset  \tr{C}_{[0,1]}$ (here we take the weight structure $w^{\pr}$ on $\cu$).

Consequently, there exist a smashing weight structure ${w^{\tr{E}}}$ on  $\tr{E}$ and a $t$-structure $t^{\tr{E}}$ adjacent to it
 such that $\pi$ is weight-exact (that is, $w^{\pr}$ descends to $\tr{E}$) and $\pi_*$ is $t$-exact (with respect to $t^{\tr{E}}$ and  $t^{can}$).

%Moreover, we have 

\item\label{hearts} %Moreover, %in case (i) (resp. (ii))
 The categories $\Ht^{\tr{E}}$ and $\hw^{\tr{E}}$ are equivalent to  
 %$\rmod_\
  ${\uu\con}$ and $\proj\, {\uu\con}$ (the subcategory of projective $\uu$-contramodules), respectively. % (resp. to $\rmod_{I-ctra}$ and~$\projr_{I-ctra}$). 

\item\label{D_u-ctra} The essential image  $\pi_*(\tr{E})$ %coincides with the subcategory % совпадает с подкатегорией
 %$D_{\uu-ctra}(\rmod) $ that 
 consists of those $\rmod$-complexes whose cohomology modules are $\uu$-contramodules.
%$D_{u-ctra}(\rmod)$ (resp.  $D_{I-ctra}(\rmod)$) that consists of those complexes whose cohomology modules are $u$-contramodules (resp. $I$-contramodules). 

\item\label{Dsq} The square 
	\begin{equation}\label{econtra}
	\begin{CD}
 \projr @>{i_{\projr}}>> \rmod\\
@VV{\pi^{\projr }}V@VV{%i_{\uu\con}^*
\Delta_{\uu}}V \\
\hw^{\tr{E}} @>{(H_{t^{\tr{E}}}^{\hw^{\tr{E}}})'}>>\uu\con
\end{CD}
\end{equation}
is commutative up to an isomorphism and the functor $(H^{t^{\tr{E}},\hw^{\tr{E}}})'$ yields an equivalence of $\hw^{\tr{E}}$ with the subcategory of projective $\uu$-contramodules. Here $i_{\projr}$ is the embedding %of the category of projective $R$-modules into 
 $\projr\to \rmod$,    $\pi^{\projr }$ is the restrictions of  $\pi$ to  $\projr$ (with the corresponding decreased target),  $(H_{t^{\tr{E}}}^{\hw^{\tr{E}}})'$ is the composition of the  corresponding restriction of $H_{t^{\tr{E}}}$ to $\hw^{\tr{E}}$ with the  aforementioned  equivalence $\Ht^{\tr{E}}\to
{\uu\con}$, and %$\tr{H\pi}_*^*$ 
%$i_{\uu\con}^*$ 
 $\Delta_{\uu}$ is the left adjoint to the %restriction of $\pi_*$ to $\Ht^{\tr{E}}$
 embedding $\uu\con\to \rmod$  (cf. Remark \ref{rpos}(\ref{iorig}) below).  %$\projr$. % (see Proposition \ref{base_t-str}(\ref{pi_*(Ht)},\ref{tadj})).
\end{enumerate}
\end{theorem}

\begin{proof}
\ref{w_trE}. Since $\tr{C}$ is well generated and $\mathcal{U}$ is a set  (both in case (i) and case (ii)), this is a particular case of Theorem \ref{main}(\ref{Brown_exists}).

\ref{ip01}.
%\ref{hearts}. %Now let us prove that $\mathcal{U}\subset  \tr{C}_{[0,1]}$. 
 In both cases $\mathcal{U}\subset  \tr{C}_{t^{can}=0}\subset \tr{C}_{w^{\pr} \ge 0}$; thus %it remains to 
 we should prove $\mathcal{U}\subset \cu_{w^{\pr} \le 1}$.

According to Proposition \ref{w^PP}(\ref{pd=wd}), the latter %fact is valid 
 property is fulfilled whenever any $U\in \mathcal{U}$ %$\tif$
 is of projective dimension at most $1$ (over $R$). 
In case (i) this %assertion %is an immediate consequence of the Поскольку 
% fact follows from the assumption $\mathrm{pd}_R U\leq 1,$ 
 is just our assumption on  $U$, and in case (ii) it is (well-known and) given by Lemma 1.9 of \cite{Pos18}. %well-known; cf. Lemma 6.9 of \cite{DwG02}?! Theorem 1.1 of \cite{HHT05}: regular elements! %cf. the proof of  \cite[Lemma 2.1(b)]{Pos17}??!

Now we can apply the %main???
  results of \S\ref{sweadjt}. %Theorem \ref{main}(\ref{S_0}) (we
	According to Proposition \ref{p01}(1), %the set
	 the elements $\mathcal{U}$ can be presented as cones of a set of $\hw^{\pr}$-morphisms.
 %We take $\mathcal{U}_1$ and $\mathcal{U}_2$ to be empty. 
 We apply  Theorem  \ref{main}(\ref{S_0}) to our $\tr{C}$ and $\uu$ (thus we set $\mathcal{U}_1$ and $\mathcal{U}_2$ to be empty).
Theorem  \ref{main}(%\ref{S_0}.
\ref{www}) gives the existence of $w^{\tr{E}}$,  part %\ref{S_0}.
 \ref{t,t_E} of that theorem gives the existence of   $t^{\tr{E}}$.
The functor $\pi_*$ is $t$-exact by Theorem  \ref{pi_&_G}(%\ref{ImGw}.
\ref{G_t-exact}).

  \ref{hearts}. Theorem  \ref{main}(%\ref{S_0}.
  \ref{pi_*(Ht_E)}) implies that  $\pi_*$ restricts to an equivalence $\Ht_E%\cong
   \to \Ht^{can}\cap{\uu}^\perp\cap {\uu}[-1]^\perp$.
 Since we have an equivalence $H^0:\Ht^{can}\cong \rmod$, for any $U\in \mathcal{U}$ we have
 $$
 D(\rmod)(U,N)\cong\rmod(U,N),\enskip D(\rmod)(U[-1],N)\cong\mathrm{Ext}^1_R(U,N);
 $$
 hence the composition $H^0\circ \pi_*$ gives an equivalence $\Ht^{\tr{E}}\to%\rmod_{\uu-ctra}
\uu\con$. %It remains Осталось применить пункт
 %Applying   Proposition \ref{Ht?Add}(\ref{iht1}.\ref{ProjHt})  we obtain $\hw^{\tr{E}}\cong\projr_{u-ctra}$. 

Assertion \ref{D_u-ctra} is a straightforward combination of the preceding ones with Theorem \ref{pi_&_G}(\ref{u012}.ii). % {H^t(N[i])}).
 %Применяя предыдущий пункт получаем, что $H^t(X^\bullet[i])\in G(\Ht^{\tr{E}})$ равносильно $H^i(X^\bullet)\in\rmod_{u-ctra}$. 

Lastly, the functor $(H^{t^{\tr{E}},\hw^{\tr{E}}})'$  yields an equivalence of $\hw^{\tr{E}}$ with the subcategory of projective $\uu$-contramodules according to %{ProjHt}
Proposition \ref{Ht?Add}(%\ref{iht1}.
\ref{ProjHt})  (along with Lemma \ref{ll168}(2)). 
 Combining this statement with  Theorem \ref{pi_&_G}(%\ref{ImGw}.
 \ref{icomm}) we obtain assertion \ref{Dsq}.
\end{proof}

\begin{note}\label{rpos}
\begin{enumerate}
\item\label{idwg} The authors certainly do not claim that this theorem is completely new. However, it is rather difficult to avoid weight structures both in the formulation and in the proof, and translating weight-exactness into the "weight structure-free" language of existing literature  requires considerable effort.%\footnote{The authors will probably work on relating \cite{DwG02}  
%it appears that existing literature does not relate contramodules to anything like weight-exact localizations.
%derived isomorphism?!  [BaP20, Th.7.1(b)].

One of the most "promising" papers %here 
 related to our one is \cite{DwG02}, %?????????????? 
and the authors will probably work on generalizing some of its results %(and combining 
 (possibly, applying Theorem \ref{tp01ext} below in the process). We will not say much on ibid. here; yet note that
Propositions 5.2 of ibid. essentially relies on a connectivity assumption (cf. Lemma \ref{ll168}(3) and Proposition \ref{hprojr}(1))
 and its conclusion is closely related to Theorem \ref{thea}(\ref{iessi}). 

Now, 
 $I$-contramodules are the %{\it $A$-complete} 
 {\it $I$-complete}  (or  $R/I$-complete) $R$-modules  in the language of ibid. (see Remark 6.2 of ibid.; it is assumed that $I$ is a finitely generated ideal in a commutative ring $R$).  This terminology justifies the name of our paper.

 %The interested reader may also compare 
 %We recommend the interested reader to look at Propositions 5.2 and 6.10 and Lemma 6.9 of ibid. to  trace %understand %?????????the relation of ibid. to our definitions and result.
%Example 3.3: more interesting SH-applications???

\item\label{iorig} %Yet we 
 Let mention some other related statements and %definitions????
 constructions here. In the setting (i) the localization functor $\pi$ was essentially constructed in Theorem 6.1(b) of  \cite{BaP20}, and in the setting (ii) it is given by Theorem 3.4 of \cite{Pos16} whenever %the ideal
  $I$ is finitely generated.  %Theorem 3.4(b) of {Pos18} for localizations?!
%Recall that
 Next,  the  left adjoint functor $\Delta_u$ to the embedding  $\rmod_{u-ctra}\to \rmod$ (cf. part \ref{Dsq} of our theorem) was provided by Proposition 3.2(b) of \cite{BaP20}; %The
 the  left adjoint $\Delta_I$ to the  embedding  $\rmod_{I-ctra}\to \rmod$ was explicitly constructed in Theorem 7.2 of \cite{Pos17} (cf. also the text preceding Remark 7.4 of ibid.) in the case where %the ideal 
 $I$ is finitely generated, and mentioned in Remark 7.7 of ibid. in the general case. %Examples 4.1(3) of {PoR17}?! {Pos18}

%\item\label{ideriso} Theorem 2.9 of {Pos16}????!!!!

%1. This theorem does not appear to contain any new results; 
%It appears that our theorem does not contain really much new information (note however that translating it into weight structure-free language of Postiselski and other authors requires considerable effort). 
%we put it here just 
% Its main roles to illustrate the general theory and to "prepare" to Corollary \ref{contra} below. 
%\item\label{idelta}

\item\label{ivar}
%It appears to be interesting to 
The authors believe that one can construct quite interesting weight-exact localizations via taking the set $\uu$ %distinct from
 that does not belong to any of the two types treated in our theorem. %Computational purposes??
In particular, it would be interesting to take $\uu=\{U_i\}$, where  $U_i$ are the codomains of some homological epimorphisms from $R$ %are as in 
 (cf. case (i) of our theorem). 

Another idea is to consider the localization functor $\tr{E}'=\tr{C}/\tr{D}'\to \tr{E}$; here one may take $\tr{C}$ and $\tr{E}$ as in (version (ii) of) Theorem \ref{tloc} and $\tr{D}'=\langle \mathcal{U}'  \rangle^{cl}_{\tr{C}}$ for the sets  $ \mathcal{U}'\subset  \mathcal{U}$ %, here the sets 
 that correspond to ideals $I'\subset I \subset R$. This functor is easily seen to be weight-exact with respect to the weight structures descended from $w^{\pr}$ (see Remark \ref{rdescends}).

One can certainly combine this observation with Theorem \ref{tp01ext} below. %???? to obtain a.
\end{enumerate}
%Certainly, we do not prove all known properties in our theorem?! In particular, derived category of contramodules?! 
\end{note}

\subsection{On "comparing heart-similar categories" and on contramodules in "categories with projective hearts"}\label{scomphearts}

\begin{theorem}\label{tp01ext}
Assume that $w$ and $w'$ are smashing weight structures on well generated triangulated  categories $\tr{C}$ and $\tr{C}'$, respectively, %respectively, there is an equivalence 
 an additive functor $F:\hw\to \hw'$ is an equivalence, a set $U$ equals $\mathcal{U}_0\cup \mathcal{U}_1 \cup \mathcal{U}_2$, where
 $\mathcal{U}_0\subset \tr{C}_{[0,1]}$,   $\mathcal{U}'_0=\cup_{U\in \mathcal{U}}\tif([U])$ (see Proposition \ref{p01}(2)),  %the classes 
 $\mathcal{U}_1$ (resp. $\mathcal{U}_2$) are  some sets of left (resp. right) $w$-degenerate objects in $\tr{C}$,  and 
$\mathcal{U}'_1$ (resp. $\mathcal{U}'_2$) are  some sets of left (resp. right) $w'$-degenerate objects in $\tr{C}'$ (cf. Theorem \ref{pi_&_G}(\ref{u012})).

Moreover, we take  %$\mathcal{U}=\mathcal{U}_0\cup \mathcal{U}_1 \cup \mathcal{U}_2$,  
 $\mathcal{U}'=\mathcal{U}'_0\cup \mathcal{U}'_1 \cup \mathcal{U}'_2$, %$\mathcal{D}=\langle\ \mathcal{U}_0\cup \mathcal{U}_1 \cup \mathcal{U}_2\rangle^{cl}_{\tr{C}}$ and $\mathcal{D'}=\langle\ \mathcal{U}'_0\cup \mathcal{U}'_1 \cup \mathcal{U}'_2\rangle^{cl}_{\tr{C}'}$, 
 $\mathcal{D}=\langle \mathcal{U}   \rangle^{cl}_{\tr{C}}$ and $\mathcal{D}'=\langle \mathcal{U}'   \rangle^{cl}_{\tr{C}'}$.

1.  %Under our assumptions, 
 Then the Verdier localizations $\tr{E} =\tr{C}/\tr{D}$ and $\tr{E}' =\tr{C}'/\tr{D}'$ are locally small,  $w$ and $w'$ descend to $\tr{E}$ and $\tr{E}'$, respectively,  and there  exist $t$-structures $t$, $t'$, $t^{\tr{E}}$,  and $t^{\tr{E}'}$ adjacent to weight structures $w$, $w'$,$w^{\tr{E}}$, and $w^{\tr{E}'}$. % on the corresponding

%Moreover, %?!there exist $t$-exact functors $\pi_*$ and $\pi'_*$ right adjoint to $\pi$ and $\pi'$, respectively, and 
2. The functor $F$ can be extended to an %isomorphism 
 equivalence of the corresponding diagrams $Sq(\pi,w)$ and  $Sq(\pi',w')$ provided by Theorem \ref{pi_&_G}(%\ref{ImGw}.
 \ref{icomm}), that is, these diagrams along with $F$ fit into a cubical diagram of functors that commutes up to isomorphisms %and %draw???!!!! Shamov?!!
such that all the remaining arrows are equivalences. 

3. Assume that both $\mathcal{U}'_1$ and $\mathcal{U}'_2$ are empty. Then  $M\in\obj\tr{C}'$ belongs to $\pi'_*(\tr{E}')$ if and only if $H_{t'}(M[i])\in \pi'_*(\Ht^{\tr{E}'})$ for all $i\in \z$. 
\end{theorem}
\begin{proof}
%Follows from
1. It suffices to apply Theorem \ref{main}(\ref{Brown_exists},%\ref{S_0}.
\ref{www}--\ref{t,t_E}) to the triples $(\cu,w,U)$ and 
$(\cu',w',\mathcal{U''})$; here $\mathcal{U''}$ is a skeleton of $\mathcal{U}'$ (thus $\mathcal{U''}$ is a set). 

2. %Assume that $U_0=\co(s)$ for $s\in S_0\subset\operatorname{Mor}\hw$. 
 Choose a set $S_0$ of $\hw$-morphisms such that $\mathcal{U}_0$ consists of their cones; see Proposition \ref{p01}(1). Proposition \ref{p01}(2)
implies that it suffices to verify that the diagram $Sq(\pi,w)$ can be described in terms of $(\hw,S_0)$ (up to an equivalence of the sort that we specified in the formulation of this assertion) completely. Now, the functor $H_t^{\hw}$ is equivalent (in the same sense) to the Yoneda embedding $\hw\to \adfupr (\hw^{op},\ab)$; see Proposition \ref{Ht?Add}(%\ref{iht1}.
\ref{i25},\ref{Ht=Add}). %?????!!
%By the same reason

Next,  Theorem \ref{pi_&_G}(\ref{u012}.i) %{main}(%\ref{u012}.(\ref{U=u0u1u2}) %(\ref{u012}(i)) 
describes in terms of $(\hw,S_0)$ the essential image of the fully faithful restriction %$\tr{H\pi}_*$ 
 $\pi_*^{\hw^{\tr{E}}}$ of $\pi_*$ to %$\hw^{\tr{E}}$
    $\Ht^{\tr{E}}$ (see Theorem \ref{pi_&_G}(\ref{ImG},%\ref{ImGw}.
\ref{G_t-exact}); hence the %corresponding
 functor %$\tr{H\pi}_*^*$ %why?! A skeleton?! Ask?!    $\pi_*^{\hw^{\tr{E}}*}$ 
 $\tr{H\pi}_*^*$  left adjoint to it  is determined by $(\hw,S_0)$ uniquely up to an equivalence. Next, the functor $H_{t^{\tr{E}}}^{\hw^{\tr{E}}}$ is equivalent to the embedding into the abelian category $\Ht^{\tr{E}}$ of its full subcategory of  projective objects; see  Proposition \ref{Ht?Add}(%\ref{iht1}.
\ref{ProjHt}) and Lemma \ref{ll168}(2). Lastly, recall that the functor %$H^{t,\hw}$
  $H_{t^{\tr{E}}}^{\hw^{\tr{E}}}$  is fully faithful by Proposition \ref{Ht?Add}(%\ref{iht1}.
  \ref{ProjHt}); hence the %compositions 
 functor $\pi_*^{\hw^{\tr{E}}} \circ H_t^{\hw}$ determines $\pi^{\hw}$ uniquely up to an isomorphism.

3. Similarly to the proof of assertion 1, it suffices to apply Theorem \ref{pi_&_G}%(%\ref{u012}.\ref{H^t(N[i])}) 
(\ref{u012}.ii) to the triple $(\cu',w',\mathcal{U''})$. 
\end{proof}

\begin{note}\label{rdetht}
The proof of Theorem \ref{tp01ext}(2) also demonstrates that the square $Sq(\pi,w)$  is determined (up to an equivalence) by the category $\Ht$ along with the essential image $\pi_*(\Ht^{t^{\tr{E}}})\subset \Ht$. %?! Yet weight structures are "easier to generate"?!
\end{note}

Now we %apply our propositions in the setting of 
 combine %our propositions 
 our theorem with Theorem \ref{tloc}.

\begin{corollary}\label{contra}
Assume that either (i) $u: R\to U$ is a homological ring epimorphism %of (associative unital) rings 
 and ${\mathrm{pd}_R U\leq 1}$, or (ii) $R$ is commutative and $I$ is an ideal in it. 
In %the first 
case (i) set $\mathcal{U}_0=\{U\}$ and in the second case  $\mathcal{U}_0=\{R[s^{-1}]:\ s\in I\}$.

Moreover, assume that $w'$ is a smashing weight structure on a well generated category $\tr{C}'$, %respectively, there is an equivalence 
 %and additive functor $F:\projr\to \hw'$ is an equivalence,
the category $\hw'$ is equivalent to $\projr$  and consider the function $\tif$ corresponding to the equivalence %of hearts
 $\hw^\pr\to \hw'$ (that is, we take $\cu=D(\rmod)$, $w=w^\pr$, and apply Proposition \ref{p01}).  
Next, take    $\mathcal{U}'_0=\cup_{U\in \mathcal{U}}\tif([U])$,  %the classes 
and choose $\mathcal{U}'_1$ (resp. $\mathcal{U}'_2$) to be  some sets of left (resp. right) $w'$-degenerate objects in $\tr{C}'$, $\mathcal{D}'=\langle \mathcal{U}'   \rangle^{cl}_{\tr{C}'}$.

\begin{enumerate}
\item\label{w_trC} Then there exists a $t$-structure $t'$ on   $\tr{C}'$ adjacent to $w'$, and $\Ht'\cong\rmod$. 

\item\label{w_trEc}  The Verdier localization $\tr{E}'=\tr{C}'/\tr{D}'$ is  locally small and the localization functor  $\pi': \tr{C}'\to\tr{E}'$ %respects coproducts and 
 possesses an exact right adjoint  $\pi'_*$. 

\item\label{w_Ec} $w'$ descends to a smashing weight structure $w^{\tr{E}'}$ on  $\tr{E}'$, and there exists a $t$-structure $t^{\tr{E}'}$
adjacent to it.

 Moreover, the category $\Ht^{\tr{E}'}$  is equivalent to the category of  $u$-contramodules in case (i) and to the category of $I$-contramodules in case (ii), and 
$\hw^{\tr{E}'}$ is equivalent to the  corresponding subcategory of projective contramodules. %category of projective $u$-contramodules in case (i) and to the category of projective $I$-contramodules in case (ii). 

\item\label{isqcupr}
There is a square 
$$\begin{CD}
 \hw' @>{H_{t'}^{\hw'}}>> \Ht'\\
@VV{\pi'^{\hw'}}V@VV{\tr{H\pi}'{}_*^*}V \\
\hw^{\tr{E}'} @>{H_{t^{\tr{E}'}}^{\hw^{\tr{E}'}}}>>\Ht^{\tr{E}'}
\end{CD}$$
of functors that is commutative up to an isomorphism; here $\tr{H\pi}'{}_*^*$ is  the left adjoint to the restriction of $\pi'_*$ to $\Ht'$ and the remaining functors are certain restrictions; see Theorem \ref{pi_&_G}(\ref{icomm}) for more detail. Moreover, this diagram is equivalent to the square (\ref{econtra}) in Theorem \ref{tloc}(\ref{Dsq})
(in the sense specified in Theorem \ref{tp01ext}(2)).

\item\label{iempty} Assume that both $\mathcal{U}'_1$ and $\mathcal{U}'_2$ are empty. Then $M\in\obj\tr{C}'$ belongs to $\pi'_*(\tr{E}')$ if and only if $H_{t'}(M[i])\in \pi'_*(\Ht^{\tr{E}'})$ for all $i\in \z$. 
 \end{enumerate}
\end{corollary}
\begin{proof}
Theorem \ref{tloc}(\ref{ip01})
 and Proposition \ref{hprojr}(1) allow us to deduce all the statements in question from Theorems \ref{tloc} and \ref{tp01ext}; here we set $\cu=D(\rmod)$, $w=w^\pr$.
\begin{comment}
\ref{pi_*}. По теореме~\ref{main}(\ref{Brown_exists}) $\tr{E}$ локально мала
 и существует правый сопряжённый функтор $\pi_*$. Также, в категории~$\tr{C}$
 выполняется свойство представимости Брауна. Поэтому можно воспользоваться
 теоремой~\ref{Ht?Add}(\ref{Ht=Add}), чтобы получить $t$-структуру $t$. 
%Ядро $\Ht$ эквивалентно $\adfupr (\hw^{op},\ab)\cong\adfupr(\projr^{op},\ab)\cong\rmod$ (см. следствие~\ref{Ht^can}).

\ref{w_E}. Данные структуры существуют в силу теоремы~\ref{main}(\ref{t,t_E}).

\ref{u/I}. По теореме\ref{main}(\ref{pi_*(Ht_E)}), $\Ht^{\tr{E}}$ эквивалентна
 подкатегории таких $N\in\Ht$, что $H_N(p)=\tr{C}(-,N)$ обратим. Если
 $\mathcal{U}$ первого вида, то данная подкатегория вычисляется аналогично
 теореме~\ref{loc}(\ref{hearts}) и эквивалентна~$\rmod_{u-ctra}$. Если $\mathcal{U}$
 второго вида, то $\pi_*(\Ht^{\tr{E}})$ эквивалентно пересечению подкатегорий
 $s_j$-контрамодулей, т.е. в точности подкатегории $I$-контрамодулей.
 
В свою очередь, $\hw^{\tr{E}}\cong\mathrm{Proj}\Ht^{\tr{E}}$ by Proposition \ref{Ht?Add}(%\ref{iht1}.
\ref{ProjHt}).%по пункту~\ref{ProjHt_E}.
\end{comment}
\end{proof}

%Finally, let us describe a rich family of examples for (the setting of) Corollary \ref{contra}.
\begin{note}\label{rcontra}
It certainly makes sense to combine Corollary \ref{contra} with the information on contramodules and on the corresponding weight-exact localizations of $D(\rmod)$ (as discussed in Theorem \ref{tloc}) that is available in the literature; cf. Remark \ref{rpos}(\ref{idwg}, \ref{iorig}).

\end{note}

\section{Some examples}\label{Examples}
In \S\ref{u:R->SR} we recall that $u:R\to U$ is a homological ring epimorphism if $U$ is an Ore localization ${\SR}$ of $R$, and  then  ${\mathrm{pd}_R S^{-1}R\leq 1}$ if ${\SR}$ is generated by a countable set of its elements as a right $R$-module (and $S$ does not contain any zero divisors).

In \S\ref{scount} we describe a rather simple example that demonstrate that the countability and the projective dimension $\le 1$ assumptions cannot be %omitted 
 lifted (if we want $w^{\pr}$ to descend to $\eu$).

In \S\ref{Stable} we discuss monogenic stable homotopy category examples to our main results. In particular, we demonstrate that some of them generalize certain statements from \cite[\S2]{Bou79}.

\subsection{Ore localizations as homological ring epimorphisms}\label{u:R->SR}

In this section we assume $S$ to be a multiplicative set in a ring $R$ such that the following {\it left Ore} condition is fulfilled:
for any $s\in S$ and $r\in R$ there exist  $s'\in S$ and $r'\in R$ such that $r's=s'r$.
 %и не содержит (левых и правых) делителей нуля.

Now let us relate Ore localization to the setting of  version (i) of Theorem \ref{tloc} and corollary \ref{contra}

\begin{proposition}\label{pore}%[{\cite[Th.1.1]{HuHeTr}}]
\begin{enumerate}
\item \label{iore} The well known localization homomorphism $u:R\to U={\SR}$ is a homological ring epimorphism.

\item \label{pd<2} %ok; generalize?!
If $S$ contains neither left nor right zero divizors and ${\SR}$ is generated by a countable set of its elements as a right $R$-module then  ${\mathrm{pd}_R S^{-1}R\leq 1.}$
\end{enumerate}
\end{proposition}
\begin{proof}
\ref{iore}.  This statement is well known. Since all elements of $U$ are of the form $s^{-1}r$ for some $s\in S$ and $r\in R$, the homomorphism
${U\otimes_R U\to U}$ (that induced by the multiplication in $U$) is easily seen to be bijective. Next, 
$U$ is  flat as a left R-module; see Proposition 1 of \cite{Tei03} along with Remark \ref{sic!SR}(1) below. %essentially
 Consequently, we have ${\mathrm{Tor}^R_i(U,U)=0}$ for all ${i>0.}$

\ref{pd<2}. This is  Theorem 1.1 of \cite{HuHeTr}.
\end{proof}

\begin{note}\label{sic!SR}%\hspace{1cm}
%\begin{enumerate} 
1. Clearly, one can replace "left" by "right" and vice versa in all the statements of this section; this corresponds to passing to opposite rings.  In particular, in \cite{Tei03} %(essentially)
  the right Ore condition and right Ore localizations were considered.
  
2. A rich class of homological ring epimorphisms that is wider than that of Ore localizations is provided by the %so-called universal localizations  %$R\to R_\Sigma$ (see %Theorem 5.14 of  \cite{HuMaVi}).
%\S6 of \cite{KrS10}).
 non-commutative localizations (see \S\ref{shist}); note however that certain additional assumptions are needed to ensure that the corresponding ring epimorphism is homological (see Remark 6.2 of  \cite{KrS10}). 
  In particular, if  $R$ is hereditaty then %all  homological ring epimorphisms  are  universal localizations
   homological ring epimorphisms  are  precisely the  non-commutative localizations (that are also called  universal localizations); see %Theorem 5.15 
	 Theorem %B 
	 6.1 of  \cite{KrS10}. %ibid. % \cite{HuMaVi} наследственно, все гомологические эпиморфизмы ими исчерпываются (\cite[Th.5.13]{HuMaVi}).
%\end{enumerate}

\end{note}

\subsection{Projective dimension $1$ and countability is important}\label{scount}

Now we will demonstrate that in certain cases the countable generation of ${\SR}$ and its  projective dimension $\le 1$ is necessary for the weight-exactness of the corresponding $\pi$. 

% В следующем предложении кольцо $R$ предполагается коммутативным.
\begin{proposition}
Let $R$ be a domain, $Q$ be the quotient field of $R$; take $\cu=D(\rmod)$ and $\tr{D}=D(Q)\subset D(R)$. 

\begin{enumerate}
\item\label{D(Q)} Then the Verdier localization $\tr{E}=\cu/\tr{D}$  is a locally small category.

Moreover,   $\mathrm{pd}_R Q\leq 1$ and $w^\pr$ descends to $\tr{E}$  if and only if $Q$ is % at %most 
 countably generated as an $R$-module.

\item\label{K[x]} Let $K$ be a field, $n\ge 2$; set $R=K[x_1,\dots,x_n]$. 

 Then $Q$ is %at %most 
 countably generated over $R$ and $w^\pr$ descends to $\tr{E}$  if and only if  $K$ is countable.
\end{enumerate}
\end{proposition}

\begin{proof}
Firstly, note that $\tr{D}=\langle Q\rangle^{cl}_{\cu}$; thus we can apply the main statements of this paper.

\ref{D(Q)}. Hence $\tr{E}$ is locally small by Theorem \ref{main}(\ref{Brown_exists}). 

Next, if $Q$ is (at most) countably generated as an $R$-module then  $\mathrm{pd}_R Q\leq 1$ (see Proposition~\ref{pore}(\ref{pd<2})).
 Hence it remains to apply Theorem~\ref{tloc}(\ref{w_trE}) in case (i).

Conversely, assume that  $Q$ is not countably generated as an $R$-module.  Then the projective dimension of the $R$-module $Q$ is at least $2$ by Theorem 1 of \cite{Kap66}.

Let us prove that $w^\pr$ does not descend to $\tr{E}$.
Since $Q$  is indecomposable (both in $\tr{D}$ and in $\tr{C}$), by  Proposition ~\ref{C'=C/D}(\ref{primes}) it remains to check that $Q$ is neither (left or right) $w^{\pr}$-degenerate nor belongs to $\tr{C}_{[s,s+1]}$ for some $s\in\z$.
Since $w^\pr$ is non-degenerate by Proposition \ref{w^PP}(\ref{class-wPP}), it suffices to verify $Q\notin \tr{C}_{[s,s+1]}$.

Since $R\not\perp Q$,  (by the orthogonality axiom (\ref{aort}) in Definition \ref{dws})  $Q\notin \cu_{w^{\pr}\ge 1}$; thus if  $Q\in \tr{C}_{[s,s+1]}$ then $s\le 0$.
 Since $\mathrm{pd}_R Q> 1$, $Q\notin \cu_{w^{\pr}\le 1}$ by Proposition  ~\ref{w^PP}(\ref{pd=wd}). Consequently, $Q$ cannot belong to $ \tr{C}_{[s,s+1]}$ for $s\le 0$.

\ref{K[x]}. If $K$ is countable then $Q$ is countable as well. Hence $Q$ is %at %most 
 countably generated; % as an $R$-module; 
thus  $w^{\pr}$ descends to $\tr{E}$ according to assertion \ref{D(Q)}.

Conversely, if $K$ is not countable then  $Q$ is %at %most 
 not countably generated as an $R$-module; see Theorem 2 of  \cite{Kap66}.  Hence $w^\pr$ does not  descend to $\tr{E}$; see assertion \ref{D(Q)}.
\end{proof}

\subsection{Connective monogenic stable homotopy categories}\label{Stable}

Now we describe a vast and popular class of triangulated categories that fulfil the assumptions of Proposition \ref{hprojr}; see Definitions 1.1.4 and
  7.1.1 of \cite{HPS97}.

\begin{definition}\label{cmSH}%{\cite[Def.1.1.4]}
$\tr{C}$ is said to be a \textit{connective monogenic stable homotopy category} whenever
\begin{itemize}
\item \tr{C} is endowed with a closed symmetric monoidal structure, compatible with the triangulation (see  Definition A.2.1 of  \cite{HPS97} for the detail);
\item $\tr{C}$ is compactly generated by the tensor unit object~$\oo$;
\item The group $\pi_n \oo=\tr{C}(\oo[n],\oo)$ is zero for  $n<0$ (cf. Remark ~\ref{monogenic}%{Example}
(\ref{pi_nX}) below).
\end{itemize}
\end{definition}

\begin{note}\label{monogenic}\hspace{1cm}
\begin{enumerate}
\item\label{R=End} Following Proposition \ref{hprojr}, we will write  $R$ for the ring ${\End^{\tr{C}}(\oo)=\pi_0(\oo)}$. Note that this ring is well known to be commutative; see \S1.3.3.1 of \cite{Riano}. %commutative by the well known  %Eckmann–Hilton theorem.  %  $\tr{C}$. 
The existence of the $t$-structure adjacent to the weight structure class generated by $\oo$ (see Proposition \ref{hprojr}(2)) and the equivalence  $\Ht\cong \rmod$ in this case is given by Proposition 7.1.2(c,e) of \cite{HPS97}. %commutative?!
 %\end{enumerate} \end{note}

%\begin{note}\label{Example}
%Рассмотрим примеры связных моногенных стабильных гомотопических категорий (см. \cite[Ex.1.2.3(a,c,f)]{HPS97}) и обсудим связь  с уже известными результатами.
\item\label{D(R)} It is no wonder that if $R$ is commutative then $D(R)$ is a monogenic stable homotopy category with $\oo=
R$; see Example 1.2.3(c) of \cite{HPS97}.
% Example 1.2.3(f)???

\item\label{SH} Recall that the stable homotopy category  ${SH}$ is a monogenic stable homotopy category with $\oo=
S^0$ (the sphere spectrum). The  corresponding  weight structure $w^{sph}$ (that is class-generated by $\PP=\{S^0\}$) is called the {\it spherical} one in \cite{Bon21}. Since $\End^{\tr{C}}(S^0)\cong \z$, we have an equivalence (from the category of free abelian groups) $F:\operatorname{FrAb}\to \hw^{sph}$, and a quasi-inverse to this equivalence is given by the restriction to $\hw$ of the singular homology functor (cf. Theorem 4.2.1(2) of \cite{Bon21}). 
The $t$-structure adjacent to $w^{sph}$ is called the Postnikov one, and its heart consists of the Eilenberg-Maclane spectra. %Moreover, the functor $SH(S^0,-)$ gives an equivalence of this $\Ht$ with $\ab$; see Theorem 4.1.1(6), Proposition 4.3.3(2) and Remark 4.3.4(1) of \cite{Bon21}.??!!

Now let us consider a localization. For a prime $p\in\Zbb$ set $U=\Zbb[p^{-1}]$ and $G=\Zbb/p\Zbb$. Since $U$ is of $\z$-projective dimension $1$, %%
we can take %consider %$\tif(\{U\})$ 
the isomorphism class $\tif([U])$ provided by Proposition \ref{p01}; here we consider the category of free abelian groups as the heart of the projective weight structure on $D(\z\text{--}\mathrm{Mod}))$. Moreover, since the only singular homology group of $S^0$ is $\z$ in degree zero, one can easily compute the singular homology of elements of $\tif([U])$ to obtain that $\tif([U])$  is the isomorphism class of the Moore spectrum  $SU$ of the group $U$.\footnote{That is, the elements of $\tif([U])$  have zero singular homology  in non-zero degrees and $U$ in degree zero; cf. Theorem 4.2.1(2) of \cite{Bon21} for %a closely related and 
much more general calculation of a closely related sort.}% (that is, its elements have zero homology groups in non-zero degrees and $U$ in degree zero).%{p01ex}!!!!!!!!!

Let us demonstrate that  %применяя теорему~\ref{pi_&_G} %????????
we can study  the so-called  $SG$-localization construction described in \cite{Bou79} %(the alternative name for it is the $p$-completion) 
 using  Theorem \ref{tp01ext}. %our main results.  
 %(которую иногда называют <<$p$-пополнением>>), разработанную  А.К. Бусфилдом в~\cite{Bou79}.
 The objects of the subcategory ${\tr{D}=\langle \{SU\} \rangle_{SH}^{cl}}\subset SH$ were said to be \textit{$SU$-local} in loc. cit. 
By Proposition 2.4 of %\cite{Bou79}
 ibid.,  $\tr{D}$ coincides with the subcategory of  \textit{$SG$-acyclic} spectra. 
  Hence $SH/\tr{D}$ is equivalent to the subcategory of %эквивалентна подкатегории 
	\textit{$SG$-local} spectra  in $SH$ (via the corresponding functor $\pi_*$).
 %Согласно~\cite[Prop.2.5-2.6]
We obtain that a  spectrum $X\in \obj SH$ is $SG$-local if and only if its homotopy groups are $u$-contramodules (in the sense of Definition \ref{ucontra}; here $u$ is the embedding $\z\to U$); this statement is contained in Proposition 2.5  of \cite{Bou79}. Moreover, one can argue similarly in the case $U=\z[J^{-1}]$, where $J$ is a set of primes, to obtain that for $G=\bigoplus_{p\in J}\z/p\z$ a spectrum $X$ is $SG$-local if and only if  all $\pi_i(X)$ are $u$-contramodules; this fact is given by Proposition 2.6 of ibid. 
%According to Propositions 2.5 and 2.6 of \cite{Bou79}, a spectrum is $SG$-local if and only if its cohomology groups are $u$-contramodules (in the sense of Definition \ref{ucontra}; here $u$ is the embedding $\z\to U$). %????!!! (ср. с теоремой~\ref{loc}(\ref{D_u-ctra})).
%Hence one can re-prove ??? using ??????!!!
 %Hence some our results are closely related to \S2 of \cite{Bou79}.

%It is 
 Lastly, note that (in contrast to $D(\rmod)$; see Proposition \ref{w^PP}(\ref{class-wPP})) the category $SH$ does contain non-zero right $w^{sph}$-degenerate objects; see Remark 4.2.3(4) of \cite{Bon21}. Thus %it makes sense to 
 one may consider various $\uu_2'$ when applying Corollary \ref{contra} to this category. 

\item\label{pi_nX} Let $E$ be a connective commutative  $S^0$-algebra (in the sense of  \cite[Definition 3.3]{EKMM07}).
 Then the derived category $\DE$ of $E$-modules is endowed with a closed monoidal structure 
 (see %\cite[Constr.2.11, Th.7.1]{EKMM07}
 Construction 2.11 and Theorem 7.1). In these terms,  $SH\cong \mathscr{D}_{S^0}$; respectively, Proposition 4.1 of ibid. gives a functor %  \cite[Prop.4.1]{EKMM07}, существует функтор  
${-\wedge E: SH\to \DE}$ that is left adjoint to the corresponding forgetful one. Consequently,  $E$ is a unit object and compactly generates $\DE.$

Lastly, is easily seen that the homotopy groups $\pi_n(X)=\DE(E,X)$ of any $X\in\mathrm{Obj}(\DE)$ (see Definition \ref{cmSH}) are functorialy isomorphic to the stable homotopy groups of the underlying spectrum (in $SH$) since for any $n\in \z$ and  $S^n=S[n]\in SH$ we have
%Введём обозначение $S^n=S[n]\in SH.$ Легко видеть, что гомотопические группы $X\in\mathrm{Obj}(\DE)$ в смысле определения~\ref{cmSH} и как спектра в $SH$ совпадают:
%$$
\begin{multline*}
\DE(E[n],X)= \DE(S^0\wedge E[n], X) \cong\\\cong \DE(S^n\wedge E, X) \cong SH(S^n,X)=\pi_n(X).
\end{multline*}
%$$
\end{enumerate}
\end{note}

\section*{Gratitudes}
The authors are deeply grateful to A.E. Druzhinin for his critically important comments to an earlier version of this text,  and to A. Vistoli for his answer on  \url{mathoverflow.net/questions/346520}.

\end{document}